\documentclass[article]{siamart1116}
\usepackage{amsmath}
\usepackage{float}
\usepackage{diagbox}
\usepackage{bbold}
\usepackage{hyperref}

\newtheorem{lem}{Lemma}

\newtheorem{rmk}{Remark}

\newcommand{\be}{\begin{equation}}
\newcommand{\ee}{\end{equation}}

\newcommand{\Dt}{\Delta t}

\newcommand{\dx}{\Delta x}
\newcommand{\dt}{\Delta t}

\newcommand{\m}[1]{\mathbf{#1}}
\newcommand{\mA}{\m{A}}

\newcommand{\mAh}{\hat{\m{A}}}
\newcommand{\mRh}{\hat{\m{R}}}
\newcommand{\mD}{\m{D}}

\newcommand{\mT}{\m{T}}
\newcommand{\mR}{\m{R}}
\newcommand{\mhA}{\m{\hat{A}}}
\newcommand{\mhR}{\m{\hat{R}}}

\newcommand{\mI}{\m{I}}

\renewcommand{\v}[1]{\boldsymbol{#1}}

\newcommand{\vc}{\v{c}}

\newcommand{\ve}{{\mathbb{1}}}

\newcommand{\vy}{\v{y}}

\newcommand{\vt}{\v{t}}

\newcommand{\sspcoef}{\mathcal{C}}

\newcommand{\DtFE}{\Dt_{\textup{FE}}}

\newcommand{\dF}{\dot{F}}

\newcommand{\ste}{\boldsymbol{\tau}}

\renewcommand{\v}[1]{\mathbf{#1}}

\title{Two-derivative  error inhibiting schemes with post-processing}

\author{%
Adi Ditkowski\thanks{School of Mathematical Sciences, Tel Aviv University, Tel Aviv 69978, Israel. 
Email: adid@post.tau.ac.il} \and
Sigal Gottlieb\thanks{Mathematics Department, University of Massachusetts Dartmouth, 285 Old Westport Road,
North Dartmouth MA 02747. Email: sgottlieb@umassd.edu} \and
Zachary J. Grant\thanks{Department of Computational and Applied Mathematics, Oak Ridge National Laboratory, Oak Ridge TN 37830. Email: grantzj@ornl.gov} 
}

\begin{document}
\maketitle


\bibliographystyle{siam}

\begin{abstract}  
High order methods are often desired for the evolution of ordinary differential equations,
 in particular those arising from the semi-discretization of partial differential equations.
In prior work in we  investigated the interplay between the local truncation error and 
the global error  to construct {\em error inhibiting } general linear methods (GLMs)
that control  the accumulation of the local truncation error over time. Furthermore we defined sufficient
conditions that allow us to post-process the final solution and obtain a solution that is 
two orders of accuracy higher than expected from truncation error analysis alone.
In this work we extend this theory to the class of two-derivative GLMs. 
We define sufficient conditions that control the growth of the error so that the solution is
one order higher than expected from truncation error analysis, and  furthermore define the
construction of a simple post-processor that will extract  an additional order of accuracy.
Using these conditions as constraints, we
develop an optimization code that  enables us to find explicit two-derivative methods up to eighth order
that have favorable stability regions, explicit strong stability preserving methods up to 
seventh order,  and A-stable implicit methods up to fifth order.
We numerically verify the order of convergence of a selection of these methods, and the
total variation diminishing performance of some of the SSP methods.
We confirm that the methods found perform as predicted by the theory developed herein.
 \end{abstract}

\section{Introduction\label{sec:intro}}

The celebrated Lax-Richtmeyer  equivalence theorem (\cite{lax1956survey}, \cite{gustafsson1995time}, \cite{quarteroni2010numerical}) 
states that if the numerical scheme is stable then its global error is {\em at least } 
of the same order as its local truncation error. 
Indeed, it is the common experience that the frequently used  methods    
produce solutions that have global error of the same order as the normalized local truncation error. 
Indeed, the Dahlquist’s Equivalence Theorem \cite{Suli2003} states this fact
for linear multistep methods.
In fact, this behavior is so expected for all one-step and multistep methods
that we typically  define the order of a stable numerical method 
solely by the order conditions derived by Taylor series analysis of  the local truncation error. 
This relationship between the order of the  local truncation error and global error is also seen in finite difference schemes for partial 
differential equations (PDEs) \cite{gustafsson1995time,quarteroni2010numerical}.
Work over the past decade reminds us that,  while for a stable scheme the global error must  {\em at least } of the same order as its local truncation error, 
 it may in fact be of higher order.
 
In this paper we consider numerical solvers for  ordinary differential equations (ODEs) of the form
\begin{eqnarray*}
& & u_t =   F(t,u)  \;,\;\;\;\;\;  t \ge 0  \\
& & u(t_0) =u_0 .\; \nonumber
\end{eqnarray*}
(where we assume  $F(t,u)$ is at least continuous with respect to $t$ 
and Lipschitz-continuous with respect to $u$, and typically much smoother).
Without loss of generality \cite{JackiewiczBook}, we focus our attention on the autonomous ODE 
\begin{eqnarray}\label{ODE}
& & u_t =   F(u)  \;,\;\;\;\;\;  t \ge 0  \\
& & u(t_0) =u_0 .\; \nonumber
\end{eqnarray}

The canonical numerical method for this problem is the  forward Euler method
\[ v_{n+1} = v_n + \Delta t F( v_n) \; , \] 
where $v_n$ approximates the exact solution at time $t_n$ and we let $v_0=u_0$. 
The forward Euler method has local truncation error $LTE^n$ and its approximation error  $\tau^n$ at any given time 
$t_n$ defined by
\[ \tau^n  = \dt LTE^n =   u(t_{n-1})  +  \Delta t F(u(t_{n-1}) )  -   u(t_{n}) \approx O(\Delta t^2) \]
and it produces a global error  which is first order accurate
\[ e^n = v_n - u(t_n)   \approx O(\Delta t) .\]
To develop methods that have higher order, 
we can add steps and define a linear multistep method \cite{butcher2008numerical},
use multiple stages as in a  Runge--Kutta methods \cite{butcher2008numerical}, 
or  include higher derivatives as we do in Taylor series methods \cite{butcher2008numerical}.
We can also combine the approaches above: the combination of multiple steps and stages is commonly see as in the  general linear methods described in \cite{butcher1993a,JackiewiczBook}, and multiple derivatives and stages have also been used 
\cite{tsai2010, MSMD, MSMD-TS, mitsui1982, Nguyen-Ba2010, ono2004, sealMSMD2014}. 
In all these cases the goal is  to increase the order of the local truncation
error and therefore of the global error. All these approaches typically produce methods that have
 global errors that are  of the same order as their local truncation errors.

It is, however, possible in some cases to devise schemes that have global errors that are {\em higher order} 
than predicted by  the local truncation errors.  Kulikov (for Nordsiek methods)  \cite{Kulikov2009} and by 
Weiner and colleagues \cite{WeinerSchmitt2009} (for  explicit two-step peer methods)
  followed along the lines of the quasi-consistency theory first introduced by Skeel in 1978 \cite{Skeel1978} to find methods that have solution of order $p+1$ 
  although their truncation errors are of order $p$. 
In  \cite{EISpaper1} Ditkowski and Gottlieb derived sufficient conditions for a family of
 general linear methods (GLMs)   of any number of steps and stages 
under which one can control the accumulation of the local truncation error over time evolution and produce methods that are one order higher than 
expected from the truncation error alone. A similar theory was developed  for implicit-explicit methods in \cite{SoleimaniWeiner2017}.
In \cite{KulikovWeiner2010}, \cite{KulikovWeiner2012}, and \cite{KulikovWeiner2018}, it was shown that under additional conditions on the method
the precise form of the first surviving term in the global error (the vector multiplying $\Delta t^{p+1} $)  can be computed 
explicitly and leveraged for error estimation.  In \cite{EIS2pp} Ditkowski, Gottlieb, and Grant showed that under less restrictive conditions the form of the 
first surviving term in the global error can be computed explicitly and exploited to enhance the accuracy to 
order $p+2$, using post-processing.

In this paper, we extend the error inhibiting approach to methods that have multiple steps, multiple stages, and two derivatives. 
In Section \ref{sec:mdEIS} we show that the theory developed in \cite{EIS2pp} carries through easily to such methods, 
and by imposing the same EIS+ conditions as in  \cite{EIS2pp}  on
on the two-derivative GLMs we consider  we can not only produce solutions with accuracy one order higher than predicted by truncation error analysis but also
precisely describe the leading  coefficients of the error term. This allows us to post-process the final time solution and obtain 
a global error that is {\em two orders higher} than the local truncation error. The construction of the post-processor is exactly 
the same as in \cite{EIS2pp} and we review it in Section \ref{sec:postproc}.
We proceed to devise error inhibiting two-step GLMs this new EIS approach, present these methods in Section \ref{newEISMDmethods},
 and in Section \ref{sec:test} we test a selection of these methods on  numerical examples 
to demonstrate their enhanced accuracy properties.

 \section{Two-derivative methods}
 Multiderivative Runge--Kutta methods, particularly two-derivative methods, were first considered in 
\cite{obreschkoff1940,Tu50,StSt63,shintani1971,shintani1972,KaWa72,KaWa72-RK,
mitsui1982,ono2004, tsai2010},  
and later  explored for use with  partial differential equations (PDEs) 
\cite{sealMSMD2014,tsai2014, LiDu2016a,LiDu2016b,LiDu2018}.
Recent interest in exploring the strong-stability properties of multiderivative Runge--Kutta methods,
(also known as multistage multiderivative methods) is evident in \cite{MSMD, MSMD-TS,Nguyen-Ba2010}.
A 2015 work by Okuonghae and Ikhile \cite{Okuonghae2015} discussed two- and three-derivatives 
methods with multiple stages and steps. 
These  two-derivative general linear methods (GLMs) are similar to those in this work.
{\em (Also cite the paper Zack refereed and look in there for references)}

  \subsection{Essentials}  \label{sec:essentials}
 
 In this work we consider the class of two-derivative general linear methods (GLMs) of the form
 \begin{equation} \label{2Dmethods}
V^{n+1} = \mD V^n +  \Delta t \mA F( V^n) +  \Delta t \mR F( V^{n+1})
+  \Delta t^2 \mhA \dF( V^n)  +  \Delta t^2 \mhR \dF( V^{n+1}), 
\end{equation}
where $V^n$ is a  vector of length $s$ that contains the numerical solution
at times  $\left( t_n+ c_j \Delta t \right)$ for $j=1,\ldots,s$:
\begin{equation}\label{multistep_v}
V^n = \left(v(t_n+ c_1 \Delta t)) ,  v(t_{n} + c_2 \Delta t)  , \ldots ,v(t_{n} + c_{s} \Delta t )  \right )^T .
\end{equation}
The function $F(V^n)$ is defined as the component-wise function evaluation on the vector $V^n$:
\begin{equation}\label{multistep_F}
F(V^n) = \left( F( v(t_{n} + c_1 \Delta t) ) , F( v(t_{n} + c_2 \Delta t) ),  \ldots , F(v(t_{n} + c_{s}  \Delta t ) ) \right )^T ,
\end{equation}
and, similarly, the function $\dF(V^n)$ is 
\begin{equation}\label{multistep_dF}
\dF(V^n) = \left( \frac{dF}{dt} ( v(t_{n} + c_1 \Delta t) ) , \frac{dF}{dt} ( v(t_{n} + c_2 \Delta t) ),  \ldots , \frac{dF}{dt} (v(t_{n} + c_{s}  \Delta t ) ) \right )^T .
\end{equation}
For convenience, we select $c_1 =0$ so that the first element in the vector  $V^n$
approximates the solution at time $t_n$, and the abscissas are  non-decreasing 
  $c_1 \leq c_2 \leq ... \leq c_s$. 
To initialize these methods, we define the  first element in the initial solution 
vector $V^0_1= u(t_0)$ and the remaining elements  $v(t_0+ c_j \Delta t) $
are computed using some highly accurate method.

We define the corresponding projection of the exact solution of the ODE \eqref{ODE} onto the temporal grid:
\begin{equation}\label{multistep_u}
U^n = \left(  u(t_{n} + c_1 \Delta t)  , u(t_{n} + c_2 \Delta t) ,  \ldots ,u(t_{n} + c_{s}  \Delta t )  \right )^T ,
\end{equation}
where  $F(U^n)$ and $\dF(U^n)$ are  the component-wise function evaluation on the vector $U^n$.
We define the {\em global error}  as the difference between the
vectors of the exact and the numerical solutions at some time $t_n$
\begin{equation}\label{E-globalerror}
E^n = V^n - U^n .
\end{equation}

Methods of the form \eqref{2Dmethods} have a local truncation error at time $t_n$ that is related to 
the approximation error $ \ste^n$
\begin{equation} \label{localerrors}
 \Delta t \; LTE^n = \ste^n = 
\left[ \mD U^{n-1} +  \Delta t \mA F( U^{n-1}) +  \Delta t \mR F( U^{n}) \right]  -U^{n} 
 \end{equation}
where
\begin{equation} 
 \ste^n = \sum_{j=0}^{\infty} \ste^n_j  \Delta t^{j}   =
 \sum_{j=0}^{\infty} \ste_j  \Delta t^{j}  \left. \frac{d^j u} {dt^j} \right|_{t=t_n}   
  \end{equation}
where the truncation error vectors $\ste_j$ have the form
 \begin{subequations} \label{tau_j}
\begin{eqnarray} 
\ste_0 & = &  \left(  \mD -  \mI \right) \ve  \label{consistency_cond}\\
\ste_j &=&  \frac{1}{(j-1)!} \left(\frac{1}{j} \mD (\vc-\ve)^j  +\mA (\vc-\ve)^{j-1}  + (j-1) \mAh (\vc-\ve)^{j-2} \right.   \label{order_cond}\\
&+& \left.  \mR \vc^{j-1} + (j-1) \mRh  \vc^{j-2} -  \frac{1}{j} \vc^j    \right) \; \; \; \; \mbox{for j=1,2, ...}  . \nonumber
  \end{eqnarray}
  \end{subequations}
  Here, $\vc= (c_1, c_2, . . . , c_{s})^T$ is the vector of abscissas and $\ve = (1,1, . . . , 1)^T$ is the vector of ones, and
the terms $\vc^j$ are understood component-wise  $\vc^j = (c_1^j, c_2^j, . . . , c^j_{s})^T$.
For a method to have order $p$ it must satisfy the order conditions
\[ \ste_j = 0 \; \; \; \; \; j \leq p.\] 
 
Note that the form  \eqref{2Dmethods} includes  both explicit and  implicit schemes, as 
$V^{n+1}$ appears on both sides of the equation. However, if 
$\mR$ and $\mRh$ are both  strictly lower triangular the scheme is explicit.
We are only interested in zero-stable methods. A sufficient condition for this is 
that the coefficient matrix $\mD $ is a rank one matrix that has row sum one (i.e. satisfies the consistency condition).
For simplicity we assume this to be the case in this the remainder of this work.

\subsection{Preliminaries} \label{sec:preliminaries}

In this subsection we make some observations that will be useful in the remainder of the paper.
In all the following we assume sufficient smoothness of all the quantities such as $F, \dF$ etc. 
Furthermore, we assume that he order conditions  $\ste_j =0$ hold for all $j=0, ..., p$, and that we 
are in the asymptotic regime so that the errors are small and we can say 
that $\| E^n \| \leq O(\dt^p) <<1$.\footnote{We will verify that this is reasonable in Lemma \ref{Lemma3}}

\noindent{\bf Observation 1} We observe that
\begin{equation}
F(U^n + E^n)   =  F(U^n)  + F_y^n E^n  + O(\dt)  E^n ,
\end{equation}
where $F_y^n = F_y(u(t_n))$.

\begin{proof} 
This is simply due to the fact that $F$ is smooth enough that it can be expanded:
\begin{align*}
F(U^n + E^n) & =    F(U^n) + 
\left(
\begin{array}{l}
F_y(u(t_n+ c_1 \Delta t)) e_{n+c_1} \\
F_y(u(t_n+ c_2 \Delta t)) e_{n+c_2} \\
\vdots \\
F_y(u(t_n+ c_{s} \Delta t)) e_{n+c_{s}} \\
\end{array}
\right)  + O\left( \|E^n\|^2 \right), \\
 & =    F(U^n) + 
\left(
\begin{array}{cccc}
F_y^{n+ c_1 }  & 0 & \cdots & 0 \\
0 & F_y^{n+ c_2 } & \cdots & 0 \\
\vdots  & \vdots & \ddots & \vdots \\
0 & 0 & \cdots &  F_y^{n+ c_{s} } \\
\end{array}
\right) E^n  + O\left( \|E^n\|^2 \right), \\
\end{align*}
where the error vector is $E^n = \left( e_{n+c_1} , e_{n+c_2} , . . . , e_{n+c_s} \right)^T$,
and we use the notation $F_y^{n+c_j} = F_y(u(t_n+ c_j \Delta t)) $.
Each term can be expanded as
\[
 F_y^{n+ c_j  }  =  F_y(u(t_n+ c_j \Delta t))  = F_y(u(t_n)) +  c_j \Delta t F_{yy}(u(t_n))  + O(\dt^2)
\]
so that, exploiting the smoothness of $F_yy$ we have
\begin{align*}
F(U^n + E^n)  & =    F(U^n)  + \left( F_y^n I + O(\dt) \right) E^n + O\left( \|E^n\|^2 \right) , 
\end{align*}
which, due to the fact that $\|E^n\| \leq O( \dt^p)$ becomes
\begin{align*}
F(U^n + E^n)   & =  F(U^n)  + F_y^n E^n  + O(\dt)  E^n .
\end{align*}

\end{proof}

\noindent{\sc Observation 2} {\em  Given the smoothness of $\dF$, 
we observe that
\begin{equation}
\dF(U^n + E^n)   =  \dF(U^n)  + \dF_y^n E^n  + O(\dt)  E^n,
\end{equation}
where $\dF_y^n = \dF_y(u(t_n))$.}

\begin{proof} 
The proof of this observation is exactly the same as above, with the assumption that $\dF$ is smooth enough to be expanded.
\end{proof} 

We use these observations to develop an equation that describes the growth of the error:
\begin{lem} \label{Lemma1}
Given a zero-stable method of the form \eqref{2Dmethods} which satisfies the order conditions 
\[ \ste_{j} = 0 \; \; \; \; \mbox{for} \; \; j=0, . . . , p \]
and where the functions $F$ and $\dF$ are smooth, the evolution of the error can be described by:
  \begin{eqnarray} \label{error_evolution}
\left( I - \Delta t \mR  F_y^{n}  +   O(  \Delta t^2 )  \right) E^{n+1}&=&\left(  \mD +  \Delta t \mA  F_y^n + O(\Delta t^2 )  \right) E^n  
 +      \ste^{n+1} .
 \end{eqnarray} 
\end{lem}

\begin{proof}
Recall that
\begin{eqnarray*}
\ste^{n+1} = \mD U^n +  \Delta t \mA F(U^n)  + \Delta t^2 \mhA \dF(U^n)   +
\Delta t \mR F(U^{n+1})  + \Delta t^2 \mhR \dF(U^{n+1}) - U^{n+1}.
\end{eqnarray*} 
and use  this equation to subtract $U^{n+1}$ from 
\begin{eqnarray*}
V^{n+1} = \mD V^n +  \Delta t \mA F(V^n)  + \Delta t^2 \mhA \dF(V^n)   +
\Delta t \mR F(V^{n+1})  + \Delta t^2 \mhR \dF(V^{n+1}) 
\end{eqnarray*} 
obtaining
\begin{eqnarray*}
V^{n+1} -U^{n+1} &=& \mD V^n  - \mD U^n  +  \Delta t \mA  \left( F(V^n)   - F(U^n) \right)  \\
&&+   \Delta t^2 \mAh  \left( \dF(V^{n})   -  \dF(U^{n}) \right)   +\Delta t \mR \left( F(V^{n+1})    -   F(U^{n+1})   \right)   \\
 && +  \Delta t^2 \mRh  \left(  \dF(V^{n+1})    -  \dF(U^{n+1})  \right)  + \ste^{n+1} 
  \end{eqnarray*} 
  using Observations 1 and 2 we have
  \begin{eqnarray*}
E^{n+1} &=& \mD (U^n+ E^n)    - \mD U^n  
+  \Delta t \mA  \left( F(U^n+E^n)   - F(U^n) \right)  \\
&&+   \Delta t^2 \mAh  \left( \dF(U^{n}+E^n)   -  \dF(U^{n}) \right)   +\Delta t \mR \left( F(U^{n+1}+E^{n+1})    -   F(U^{n+1})   \right)   \\
 && +  \Delta t^2 \mRh  \left(  \dF(U^{n+1}+E^{n+1})    -  \dF(U^{n+1})  \right)  + \ste^{n+1}  \\
&=&  \mD E^n +  \Delta t  \left(  \mA F_y^n + O(\dt)  \right) E^n + \Delta t^2  \left( \dF_y^n  \mAh + O(\dt)   \right)  E^n   \\
&&  +  \Delta t \left( \mR F_y^{n+1} + O(\dt) \right)  E^{n+1}  +  \Delta t^2 \left(   \dF_y^{n+1}   \mRh + O(\dt)   \right)  E^{n+1} + \ste^{n+1}  ,
  \end{eqnarray*} 
so that
  \begin{eqnarray*}
 E^{n+1}&=& \mD E^n  +  \Delta t \mA  F_y^n E^n  +  \Delta t \mR  F_y^{n+1} E^{n+1}
 +  \Delta t^2 \mAh \dF_y^n E^n  +  \Delta t^2 \mRh  \dF_y^{n+1} E^{n+1}  \\
 && + O(\dt^2) E^n +  O(\dt^2) E^{n+1} + \ste^{n+1} .
 \end{eqnarray*} 
 Note that  $F_y^n, \dF_y^n, F_y^{n+1},$ and $\dF_y^{n+1}$ are all scalars.
 Now recall that $ O(\dt) O(\|E^n\|) < O(\dt^2)$ since $p \geq 1$, and that terms of the form $\mRh F^{n+1}_y$
 and $\mAh F^{n}_y$ are all bounded, so we have 
   \begin{eqnarray*}
 E^{n+1}&=& \mD E^n  +  \Delta t \mA  F_y^n E^n  +  \Delta t \mR  F_y^{n+1} E^{n+1} +  O(\Delta t^2)  E^n  +  O(\Delta t^2) E^{n+1} .
 \end{eqnarray*}
Now move the $E^{n+1}$ terms to the left side:
 \begin{eqnarray*}
\left( I - \Delta t \mR  F_y^{n+1}  +   O(  \Delta t^2 )  \right) E^{n+1}&=& \mD E^n  +  \Delta t \mA  F_y^n E^n  
 +  O(\Delta t^2 )    + \ste^{n+1} .
 \end{eqnarray*} 
 Noting that $ F_y^{n+1} =  F_y^{n} +O(\dt)$ we have  the desired result
\[\left( I - \Delta t \mR  F_y^{n}  +   O(  \Delta t^2 )  \right) E^{n+1} =\left(  \mD +  \Delta t \mA  F_y^n + O(\Delta t^2 )  \right) E^n  
 +      \ste^{n+1} .
\]
\end{proof}
 
 We now turn to verifying that the assumption we made in the first paragraph of this subsection, that 
 we can assume that the error $E^n$  is small:
\begin{lem} \label{Lemma3}
Given a zero-stable  scheme \eqref{2Dmethods}, if the order conditions
$\ste_j=0 $ are satisfied for $j=0, . . ., p$, and  $\mR$ and $F_y^n$ are bounded,
 then there is  some time interval [0,T]  such that the error $E^n$ (given by  \eqref{error_evolution} )
satisfies
\[ \| E^n \|  \leq O(\dt^p) <<1.\]
\end{lem}
The proof of this Lemma is given in \cite{EIS2pp}, Lemma 3.


 \section{Designing multi-derivative error inhibiting schemes that can be  post-processed to order $p+2$}   
In this section we show how to construct two-derivative GLMs of the form \eqref{2Dmethods} 
that are error inhibiting and so produce a solution of order $p+1$ even though one would
only expect order $p$ from their truncation errors. We further show that under additional conditions we can express the exact form of the error and define an
associated post-processor that allow us to recover order $p+2$ from  a scheme that would otherwise be only $p$th order accurate.

\subsection{Two-derivative error inhibiting schemes that have $p+1$ order} \label{sec:mdEIS}

In this section we consider a  multi-derivative GLM  of the form   \eqref{2Dmethods}
where we assume that $\mD$ is a rank one matrix  that satisfies the consistency condition 
$\mD \ve =\ve$ so that this scheme is zero-stable.
 Furthermore, we assume that the coefficient matrices  $\mD$, $\mA$, $\mR$, $\mhA$, $\mhR$ 
 are such that the  order conditions $\ste_j=0$ are satisfied for all $j \leq p$, so that the method 
 will give us a numerical solution that has error that is guaranteed of order  $p$.
In the following theorem we show that if the truncation error vector $\ste_{p+1}$ 
lives in the null-space of the operator $\mD$ the order of the error can be shown to be of order $p+1$. 
Furthermore, the theorem shows  that under additional conditions 
on the coefficient matrices  $\mD$, $\mA$, $\mR$, $\mhA$, $\mhR$,  
we can determine precisely what the leading term of this error 
will look like and therefore remove it by post-processing.

\begin{theorem}
Given a zero-stable two-derivative general linear method of the form
\begin{eqnarray*} 
V^{n+1} = \mD V^n +  \Delta t \mA F( V^n) +  \Delta t \mR F( V^{n+1})
+  \Delta t^2 \mhA \dF( V^n)  +  \Delta t^2 \mhR \dF( V^{n+1})
\end{eqnarray*}
where $\mD$ is a rank one matrix  that satisfies the consistency condition $\mD \ve =\ve$,
and the  coefficient matrices  $\mD$, $\mA$, $\mR$, $\mhA$, $\mhR$ satisfy the order conditions 
\[ \ste_j  = 0 \; \; \; \; \mbox{for} \; \;  j=1, . . . , p,\] where
\begin{eqnarray*}
 \ste_j &=&  \frac{1}{(j-1)!} \left(\frac{1}{j} \mD (\vc-\ve)^j  +\mA (\vc-\ve)^{j-1}  + (j-1) \mAh (\vc-\ve)^{j-2} \right. \\
&+& \left.  \mR \vc^{j-1} + (j-1) \mRh  \vc^{j-2} -  \frac{1}{j} \vc^j    \right) \; \; \; \; \mbox{for j=1,2, ...} ,
\end{eqnarray*}
if the error inhibiting condition 
\begin{subequations}  \label{EIS2conditions}
\begin{eqnarray}
\mD\ste_{p+1}=0  \label{con1} 
\end{eqnarray}
is satisfied, then the numerical solution produced by this method will have error \[E^n = O(\dt^{p+1}).\]
Furthermore, if the conditions 
\begin{eqnarray}
\mD\ste_{p+2}=0  \label{con2}\\
\mD(\mA+\mR)   \ste_{p+1}=0 \label{con3}
\end{eqnarray}
\end{subequations}
are also satisfied, then error vector will have the more precise form:
 \begin{equation} \label{Error_form}
E^n= \Delta t^{p+1}   \ste^n_{p+1} +   O( \Delta t^{p+2}).
\end{equation}

\end{theorem}

\begin{proof}
Using Lemma 1 we  obtain the equation for the evolution of the error
\begin{eqnarray} \label{eq:error}
E^{n+1} & = & \left( I -  \Delta t \mR F_y^n + O(\dt^2) \right)^{-1} \left[ \left( \mD +  \Delta t \mA F_y^n + O(\dt^2)  \right) E^n + \ste^{n+1} \right]
\end{eqnarray}
the fact $F$ is smooth assures us that $ \| \Delta t \mR F_y^n \|  << O(1)$ so that we can expand the first term to obtain
\begin{eqnarray*} 
E^{n+1} & = & \left( I +  \Delta t \mR F_y^n + O(\dt^2) \right)  [ \left( \mD +  \Delta t \mA F_y^n + O(\dt^2)  \right) E^n + \ste^{n+1} ] \\
& = &   \left( \mD +  \Delta t F^n_y ( \mR \mD + \mA )  + O(\dt^2) \right) E^n \\
&& +  \dt \left( \dt^p \ste_{p+1}^{n+1} + \dt^{p+1} ( \ste_{p+2}^{n+1}  + F_y^n \mR \ste_{p+1}^{n+1}) +  O(\dt^2) \right)  \\
 & = & {Q}^n E^n + \dt T_e^n .
\end{eqnarray*}
The discrete Duhamel principle (given as  Lemma 5.1.1 in  \cite{gustafsson1995time})  
states that  given an iterative process of the form
\[ E^{n+1} \,=\, Q^n E^n + \dt T_e^n \] 
where $Q^n$ is a linear operator, we have
\begin{equation} \label{App_A_110}
E^{n} \, = \,   \prod_{\mu=0}^{n-1} Q^{\mu}  E^0  \,+\,  \dt \, \sum_{\nu=0}^{n-1}  
\left( \prod_{\mu=\nu+1}^{n-1} Q^{\mu}  \right)  T_e^{\nu} \;.
\end{equation}

To analyze the error, we separate it  into four parts:
\[ E^{n} \, = \,   \underbrace{\prod_{\mu=0}^{n-1} Q^{\mu}  E^0}_I  \,+
\underbrace{  \dt T_e^{n-1}  }_{II}  + \underbrace{ \dt Q^{n-1} T_e^{n-2} }_{III}  
+  \underbrace{  \dt \, \sum_{\nu=0}^{n-3}  \left( \prod_{\mu=\nu+1}^{n-1} Q^{\mu}  \right)  T_e^{\nu}}_{IV} \]
and discuss each part separately:
\begin{itemize}
\item[I.] The  method is initialized so that the  numerical solution vector $V^0$ is accurate
 enough to  ensure that the initial error $E^0$ is negligible and we can ignore the first term. 
 
\item[II.] The final term in the summation is $\dt T_e^{n-1}$. Recall that 
\[ T_e^{n-1} =  \left( \dt^p \ste_{p+1}^{n} + \dt^{p+1} ( \ste_{p+2}^{n}  + F_y^{n-1} \mR \ste_{p+1}^{n}) +  
O(\dt^{p+2}) \right) \] 
so that   this term contributes to the final time error the term
\[ \dt  T_e^{n-1} = \dt^{p+1}  \ste_{p+1}^{n} +  O(\dt^{p+2}) .\]

\item[III.]  The term  $\dt Q^{n-1} T_e^{n-2}$ is a product of the operator
\[ Q^{n-1} =  \left( \mD +  \Delta t F^{n-1}_y ( \mR \mD + \mA )  + O(\dt^2) \right) \]
and the approximation error 
\[ T_e^{n-2} =\left( \dt^p \ste_{p+1}^{n-1} + \dt^{p+1} ( \ste_{p+2}^{n-1}  + F_y^{n-2} \mR \ste_{p+1}^{n-1}) +  O(\dt^{p+2}) \right) ,\]
so we obtain 
\[ Q^{n-1}  T_e^{n-2}  = \dt^p \mD \ste_{p+1}^{n-1}  + O(\dt^{p+1})  = O(\dt^{p+1}),\]
due to  the condition \eqref{con1} that states that $ \mD \ste_{p+1} = 0$.

Looking more closely at the $O(\dt^{p+1})$ terms that remain in this product  
and applying \eqref{con1} again
\begin{eqnarray*}
& \mD ( \ste_{p+2}^{n-1}  + F_y^{n-2} \mR \ste_{p+1}^{n-1}) 
+   F^{n-1}_y ( \mR \mD + \mA )  \ste_{p+1}^{n-1} + O(\dt)  \\
& =  \mD ( \ste_{p+2}^{n-1}  + F_y^{n-2} \mR \ste_{p+1}^{n-1}) 
+   F^{n-1}_y  \mA   \ste_{p+1}^{n-1} + O(\dt)  \\
& =  \mD ( \ste_{p+2}^{n-1}  + F_y^{n-1} \mR \ste_{p+1}^{n-1}) 
+   F^{n-1}_y  \mA   \ste_{p+1}^{n-1}  + O(\dt)  \\
& =  \mD \ste_{p+2}^{n-1}  + F_y^{n-1} ( \mD  \mR  + \mA ) \ste_{p+1}^{n-1}  + O(\dt)  \\
\end{eqnarray*}
where we used the observation that $F^{n-2}_y   = F^{n-1}_y  + O(\dt)$.
The product is then
\[ Q^{n-1}  T_e^{n-2} = \dt^{p+1} \left[ \mD \ste_{p+2}^{n-1}  + F_y^{n-1} ( \mD  \mR  + \mA ) \ste_{p+1}^{n-1}  
 \right]  + O(\dt^{p+2}) , \]
we now use the two error inhibiting conditions \eqref{con2} and \eqref{con3} to  obtain
\[ Q^{n-1}  T_e^{n-2} =  O(\dt^{p+2}) . \]

\item[IV.] Finally we look at the rest of the sum and use the boundedness of the operator $Q^n$ to observe 
\begin{eqnarray*} \label{App_A_120}
\left\| \dt \, \sum_{\nu=0}^{n-3}  \left( \prod_{\mu=\nu+1}^{n-1} Q^{\mu}  \right)  T_e^{\nu}  \right\| & = &
\left\| \dt \, \sum_{\nu=0}^{n-3}  \left( \prod_{\mu=\nu+3}^{n-1}\hat {Q}^{\mu}  \right)   \left(\hat {Q}^{\nu+2}\hat {Q}^{\nu+1} T_e^{\nu} \right) \right\|  \nonumber \\
& \leq & \dt \, \sum_{\nu=0}^{n-3}  \left\|  \prod_{\mu=\nu+3}^{n-1} Q^{\mu}  \right\|   \left\|  Q^{\nu+2}  Q^{\nu+1} T_e^{\nu}  \right\| \nonumber  \\
& \leq & \dt \, \sum_{\nu=0}^{n-3}\left( 1+ c \,\dt \right)^{n- \nu -3}   \left\|  Q^{\nu+2}  Q^{\nu+1} T_e^{\nu}  \right\| \\
& = & \frac{\exp \left( c t_n\right) -1}{c}  \max_{\nu=0, . . .,  n-3}  \left\|  Q^{\nu+2}  Q^{\nu+1} T_e^{\nu}  \right\|  . 
\end{eqnarray*}
Clearly, the first term here is a constant that depends only on the final time, so it is the product $Q^{\nu+2}  Q^{\nu+1} T_e^{\nu}$ we need to bound.
Using the definition of the operators $Q^\mu$ 
\[  Q^{\nu+2}  Q^{\nu+1} =  \left[  \mD^2  +  \Delta t
\left(  F_y^{\nu+2}   \left( \mR \mD+\mA \right) \mD +F_y^{\nu+1}  \mD \left( \mR \mD +\mA \right)  \right )
+  O(\dt^2) \right] \]
and the approximation error
\[ T_e^{\nu} = \left( \dt^p \ste_{p+1}^{\nu+1} + \dt^{p+1} ( \ste_{p+2}^{\nu+1}  + F_y^{\nu} \mR \ste_{p+1}^{\nu+1}) +  O(\dt^{p+2}) \right) \]
we have
\begin{eqnarray*}
Q^{\nu+2}  Q^{\nu+1}T_e^{\nu}
  & = &   \Delta t^{p}  \left( \mD  +  
 \Delta t  F_y^{\nu+2}  \left(  \mR \mD+\mA \right) 
 + \dt F_y^{\nu+1}  \mD \mR \right)  \mD  \ste^{\nu+1}_{p+1}  \\
  &+ &    \Delta t^{p+1} \left( F_y^{\nu+1} \mD^2  \mR   +  F_y^{\nu+1}  \mD \mA  \right) \ste^{\nu+1}_{p+1} 
  +  \Delta t^{p+1} \mD^2  \ste^{\nu+1}_{p+2} +   O(\dt^{p+2}) . 
\end{eqnarray*}
Using the fact that in our case $\mD^2 = \mD$ and that $F_y^{\nu+2}  = F_y^{\nu+1}  +O(\dt)$ we obtain
\begin{eqnarray*}
 Q^{\nu+2}  Q^{\nu+1}T_e^{\nu} & = & 
 \Delta t^{p}  \left( \mD  +  
 \Delta t  F_y^{\nu+2}  \left(  \mR \mD+\mA \right) 
 + \dt F_y^{\nu+1}  \mD \mR \right)  \mD  \ste^{\nu+1}_{p+1}  \\
  &+ &    \Delta t^{p+1}  F_y^{\nu+1}  \mD  \left(   \mR   +   \mA  \right) \ste^{\nu+1}_{p+1} 
  +  \Delta t^{p+1} \mD  \ste^{\nu+1}_{p+2} +   O(\dt^{p+2}) .
\end{eqnarray*}
The first term and third terms disappear because  of \eqref{con1} and \eqref{con2}
\[\mD \ste_{p+1} = 0 \; \; \; \; \; \mbox{and} \; \; \; \; \; \mD \ste_{p+2} =0. \] 
The second term is eliminated by \eqref{con3} 
\[ \mD  \left(   \mR   +   \mA  \right) \ste_{p+1}  = 0 .\]
So that 
\[  Q^{\nu+2}  Q^{\nu+1}T_e^{\nu}  = O(\dt^{p+2}) .\]

\end{itemize}

All these parts together  give us the result
\begin{eqnarray}
E^{n} & = &     \underbrace{\prod_{\mu=0}^{n-1} Q^{\mu}  E^0}_I  \,+
\underbrace{  \dt T_e^{n-1}  }_{II}  + \underbrace{ \dt Q^{n-1} T_e^{n-2} }_{III}  
+  \underbrace{  \dt \, \sum_{\nu=0}^{n-1}  \left( \prod_{\mu=\nu+1}^{n-3} Q^{\mu}  \right)  T_e^{\nu}}_{IV}  \nonumber \\
& =&  0 + \underbrace{  \dt^{p+1} \ste^n_{p+1}  }_{II}  +  \underbrace{ O(\dt^{p+2} )  }_{III} +  \underbrace{ O(\dt^{p+2} )  }_{IV} \nonumber  \\
& = &   \dt^{p+1} \ste^n_{p+1} + O(\dt^{p+2}).  
\end{eqnarray}

\end{proof}

\begin{rmk}
We note that generalizing this approach from our work in \cite{EIS2pp} to include a second derivative was quite simple.
This is due to the fact that the second derivative appears with a $\dt^{p+2}$ term multiplying it, so that the operator $Q^n$ is in fact exactly 
the same for a two-derivative method of the form \eqref{2Dmethods} as for the methods considered in  \cite{EIS2pp}. For this reason, expanding this approach
to higher derivatives will be simple: the only change would be in the definition of the order conditions \eqref{tau_j}.
\end{rmk}

\subsection{Designing a post-processer to recover $p+2$ order} \label{sec:postproc}

At every time-step $t_n$ the error vector $E^n$ has the particular form \eqref{Error_form}
\[ E^n= \Delta t^{p+1}   \ste^n_{p+1} +   O( \Delta t^{p+2}), \]
so that the  leading order term $\Delta t^{p+1}   \ste^n_{p+1}$ can be removed by post-processing 
 at the end of the computation and a final time solution of order $p+2$ can be obtained, as we showed in \cite{EIS2pp}.
We notice that the form of the error for multi-derivative method  \eqref{2Dmethods} is exactly the same as that of the error of the EIS methods 
presented in \cite{EIS2pp}, the only difference being the definition of the truncation error vector. 
The theoretical discussion of this post-processor is presented in  \cite{EIS2pp} and we refer the reader to the presentation there.
We review one such construction below:
\begin{enumerate}
\item Select the number of computation steps $m$ that will be used, requiring that
\[ m s \geq p+3\]
where $s$ is the number of steps and $p$ the truncation-error order of the time-stepping method.
\item Define the time vector of all the  temporal grid points in the last $m$  computation steps:
\[
\tilde{\vt} = \left (  t_{n-m+1} + c_1 \dt,       ...,   t_{n-m+1} + c_{s} \dt,  ...,
 t_{n} + c_1 \dt,       ...,   t_{n} + c_{s} \dt
 \right )^T \]
\item Correspondingly, concatenate the final $m$ solution vectors 
 \begin{equation*} 
 \tilde{V}^n \,=\, \left ( 
\begin{array}{c}
V^{n-m+1} \\ 
\vdots \\
V^{n} \\
\end{array} 
\right) \; , \; \; \;  \mbox{and}  \; \; \; \;
\tilde{U}^n
\,=\, \left ( 
\begin{array}{c}
U^{n-m+1} \\ 
\vdots \\
U^{n} \\
\end{array} 
\right).
\end{equation*}
\item Stack $m$ copies of the final truncation error vector  $\ste_{p+1}$
\[\tilde{\ste} = \left(\underbrace{\ste_{p+1}^T,...,\ste_{p+1}^T}_{m} \right)^T \]
\item Define the vertically flipped Vandermonde interpolation  matrix 
$\mT$ on the vector $\tilde{\vt} $, and replace the first column 
(the one corresponding to the highest polynomial term) by $\tilde{\ste}$:
 \[ \mT = \left( \tilde{\ste},  \tilde{\vt}^{ms-1},  \tilde{\vt}^{ms-2},
 . . ., \tilde{\vt}^{2}, \tilde{\vt}, \ve \right) \]
where terms of the form $\tilde{\vt}^q$ are understood as component-wise exponentiation.
Note that $\mT$ is a square matrix of dimension $ms \times ms$.
\item Define the post-processing filter
\[ \Phi = T \, \text{diag} \left(0, \underbrace{1,...,1}_{ms -1 },  \right)\,T^{-1}.\]
\item Finally, left-multiply the solution vector $ \tilde{V}^n $ by the post-processing filter $\Phi$  to obtain the post-processed solution
\[ \hat{V}^n = \Phi   \tilde{V}^n ,\] which will be of order $p+2$, as shown in \cite{EIS2pp}.
\end{enumerate}

As in \cite{EIS2pp} we remark that  this process may break down if the matrix $\mT$ is not invertible, and that numerical instabilities may result if 
$ \| \Phi \|$   is large. This should be verified while building these matrices.


\section{New error inhibiting GLM schemes}  \label{newEISMDmethods}
   In this section we present the new general linear methods (GLMs) that satisfy the error inhibiting
   conditions \eqref{con1}--\eqref{con3}. We divide these methods into three types: explicit methods 
   that have good linear stability regions, explicit strong stability preserving methods,
   and implicit methods that are A-stable. The coefficients of all the methods mentioned in this section
   can be downloaded from our GitHub site \cite{EISpp_MD_github}.
   
\subsection{Explicit  EIS methods with favorable linear stability regions}  \label{explicit_methods} 
We used the optimization procedure to produce a number of explicit methods 
with relatively large regions of linear stability.
As is usual in this field, the linear stability region is defined  as the value of the complex number $\lambda$ 
so that the method applied to the problem $y'=\lambda y$ (and so $y''= \lambda y' =  \lambda^2 y$) is stable.

We consider two types of explicit methods in this section, and denote them  eEIS(s,P)$_2$ and eEIS+(s,P)$_2$.  
The eEIS(s,P)$_2$ methods have $s$ stages and  satisfy the order conditions to some order $p$ 
and the  EIS condition \eqref{con1} so that  the overall order is $P=p+1$.
The eEIS+(s,P)$_2$ methods also have $s$ stages and  satisfy the order conditions to some
order $p$ but they satisfy all three   error inhibiting  conditions \eqref{con1} --  \eqref{con3} 
so that  the overall order is $P=p+2$. 
In this section we present and compare the linear stability regions of some of these methods.
Due to space constraints we do not present the coefficients of all the methods. However,
in Appendix \ref{appendix:coef} we present the coefficients of a selection of the methods that
will be used in the numerical tests in Section \ref{sec:test},
and the coefficients of all methods mentioned in this section can be downloaded  
from our GitHub site \cite{EISpp_MD_github}.

 We begin with an explicit EIS GLM that has $s=2$ stages and satisfies the truncation error 
conditions to $p=2$ and the EIS condition \eqref{con1} so that the overall order is $P=3$, 
and refer to this method as eEIS(2,3)$_2$.  In Figure \ref{EISvsKG}
we compare the stability region of this method to that of the five-stage third order Kinmark-Gray method eKG(5,3)  \cite{KinmarkGray}.
The linear stability region of the  Kinmark-Gray method eKG(5,3)  method, which has five function evaluations,
 is shown with the y-axis going from $(-5,5)$, while the eEIS(2,3)$_2$ method, 
 which has four total function evaluations 
(two evaluations of $F$ and two of $\dF$), is shown with the y-axis going from $(-2s,2s)=(-4,4)$. Clearly, 
it seems that the Kinmark-Gray method is more efficient, even considering the extra cost of computing
five function evaluations over four for the eEIS(2,3)$_2$ method. 
However, the argument for using two-derivative methods is 
that frequently the second derivative is often needed for other parts of the simulation (e.g. the spatial derivative), 
and is  computed whether or not it is used in the time-stepping. 
Thus, we can consider the cost of the eEIS(2,3)$_2$ method to be that of two function evaluations rather than four. 
In this case, the eEIS(2,3)$_2$ has a {\em larger} effective linear stability region than the eKG(5,3) method.
The coefficients of the eEIS$_2$(2,3) are given in Appendix \ref{appendix:coef}  while the coefficients of the 
well-known eKG(5,3) method  can be found in \cite{KinmarkGray} or downloaded from our GitHub site \cite{EISpp_MD_github}. 

\begin{figure}
\begin{center}
 \includegraphics[width=.485\textwidth] {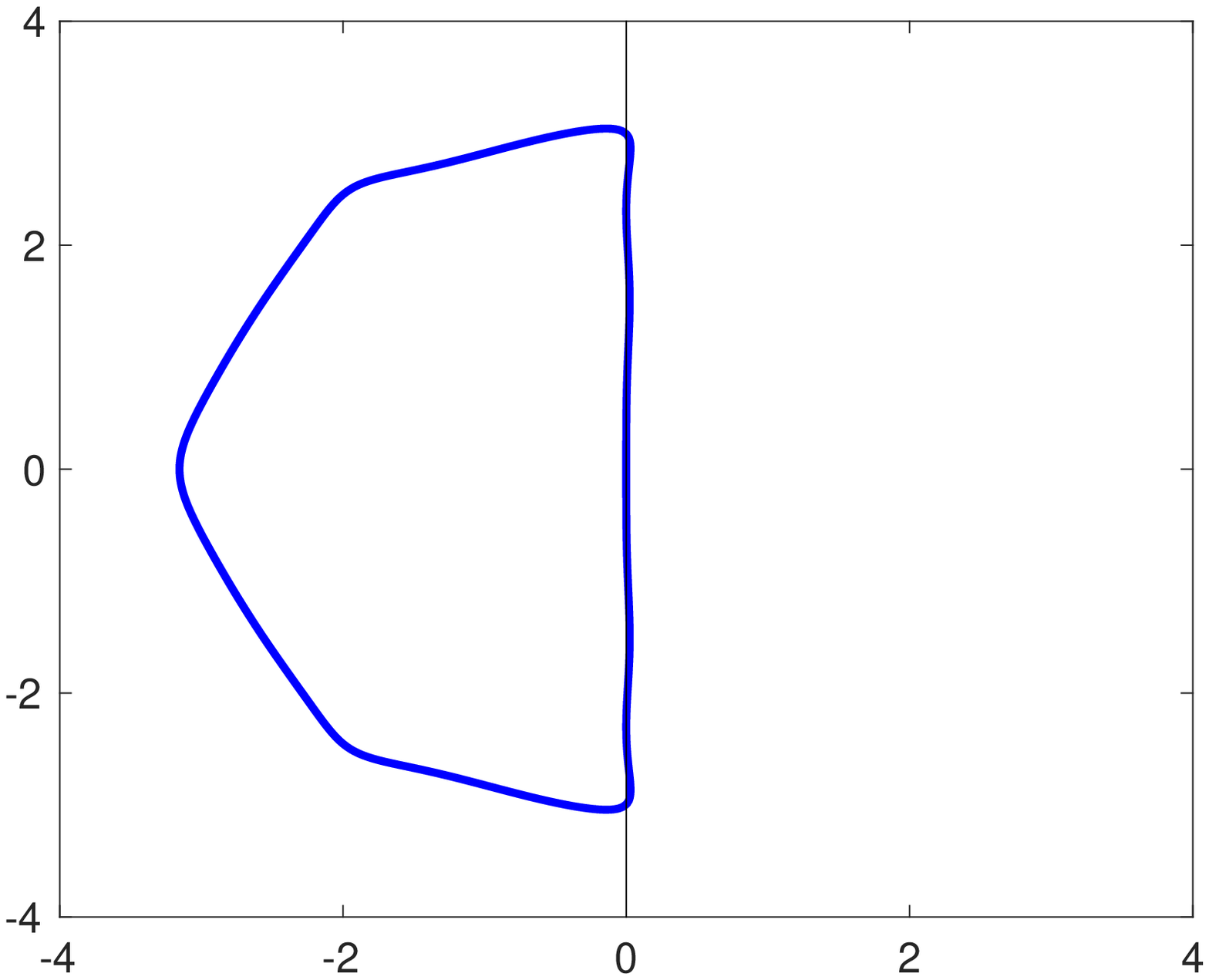} 
   \includegraphics[width=.485\textwidth]{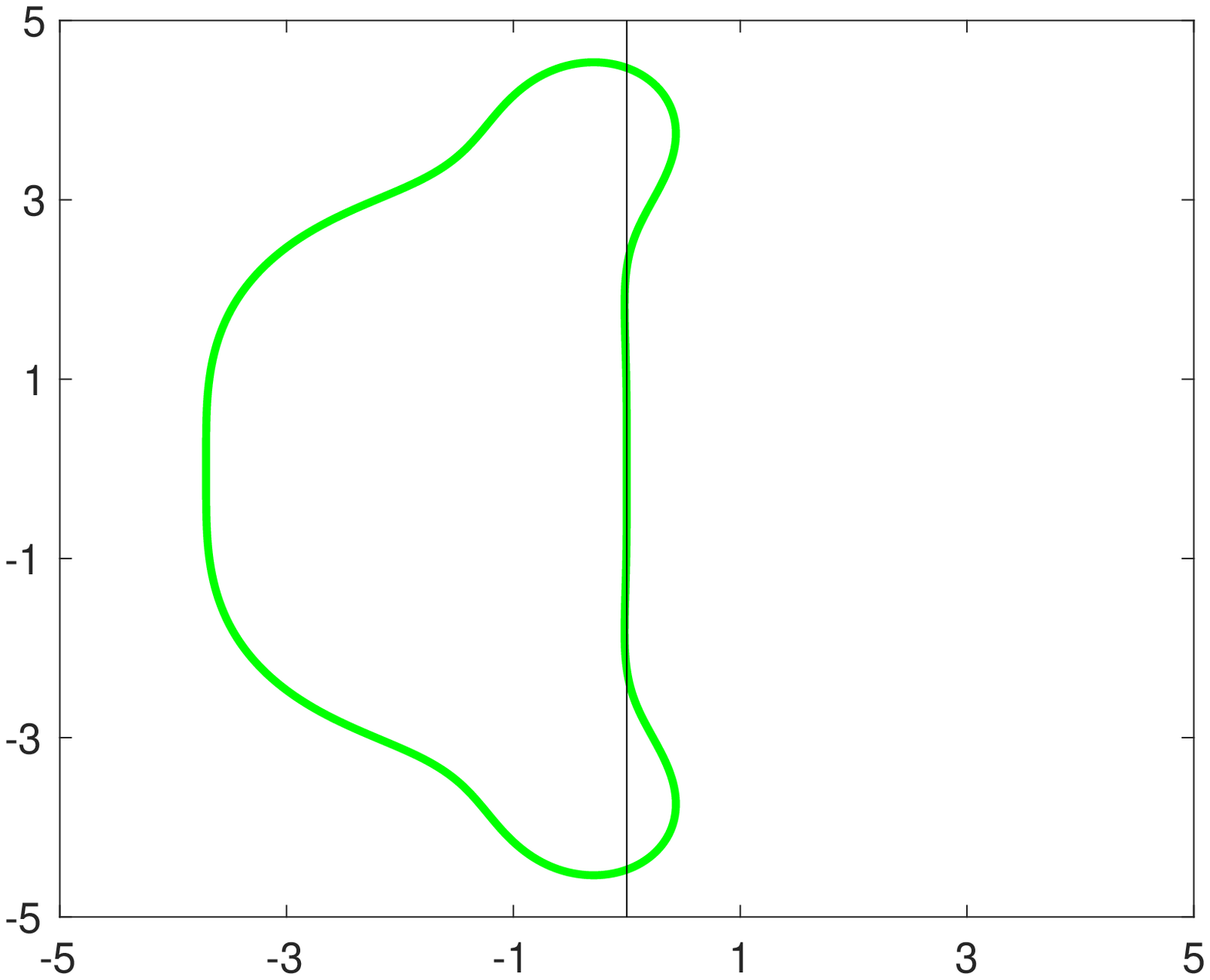} 
 \caption{\label{EISvsKG} The linear stability region of the EIS method eEIS(2,3)$_2$  on left 
 compared to that of the Kinmark-Gray method  eKG(5,3) on the right.
 }
 \end{center}
\end{figure}

Next, we compare the linear stability regions of EIS(s,P)$_2$ and EIS+(s,P)$_2$ methods we found using our optimization method. 
In Figure \ref{EISvsEIS+} on the left we show the linear stability regions of the two-stage fifth order methods
 eEIS(2,5)$_2$ method (in blue) and the eEIS+(2,5)$_2$ method (in red).  
 In Figure \ref{EISvsEIS+} on the right we show the linear stability regions of the three-stage seventh order methods
 eEIS(3,7)$_2$ method (in blue) and the eEIS+(3,7)$_2$ method (in red).
In both cases we can  see that the linear stability region of the EIS+(s,P)$_2$ is larger than that of the corresponding 
EIS(s,P)$_2$ method. This suggests that the additional order condition that the EIS(s,P)$_2$ method must satisfy (the 
EIS(s,P)$_2$ method must satisfy $P-1$ conditions instead of $P-2$ conditions for the EIS+(s,P)$_2$) play more of a role
in  constraining the linear stability region than the additional error inhibiting conditions \eqref{con2} and \eqref{con3}
that the EIS+(s,P)$_2$ conditions must satisfy.
These results show that from a linear stability perspective, it is more efficient to use the 
EIS+ methods than the corresponding EIS methods. This is because the cost of one-time
post-processing is negligible compared to the cost incurred by a restricted time-step which will 
require many more time-steps to reach the same final time.

\begin{figure}[h]
\begin{center}
 \includegraphics[width=.485\textwidth] {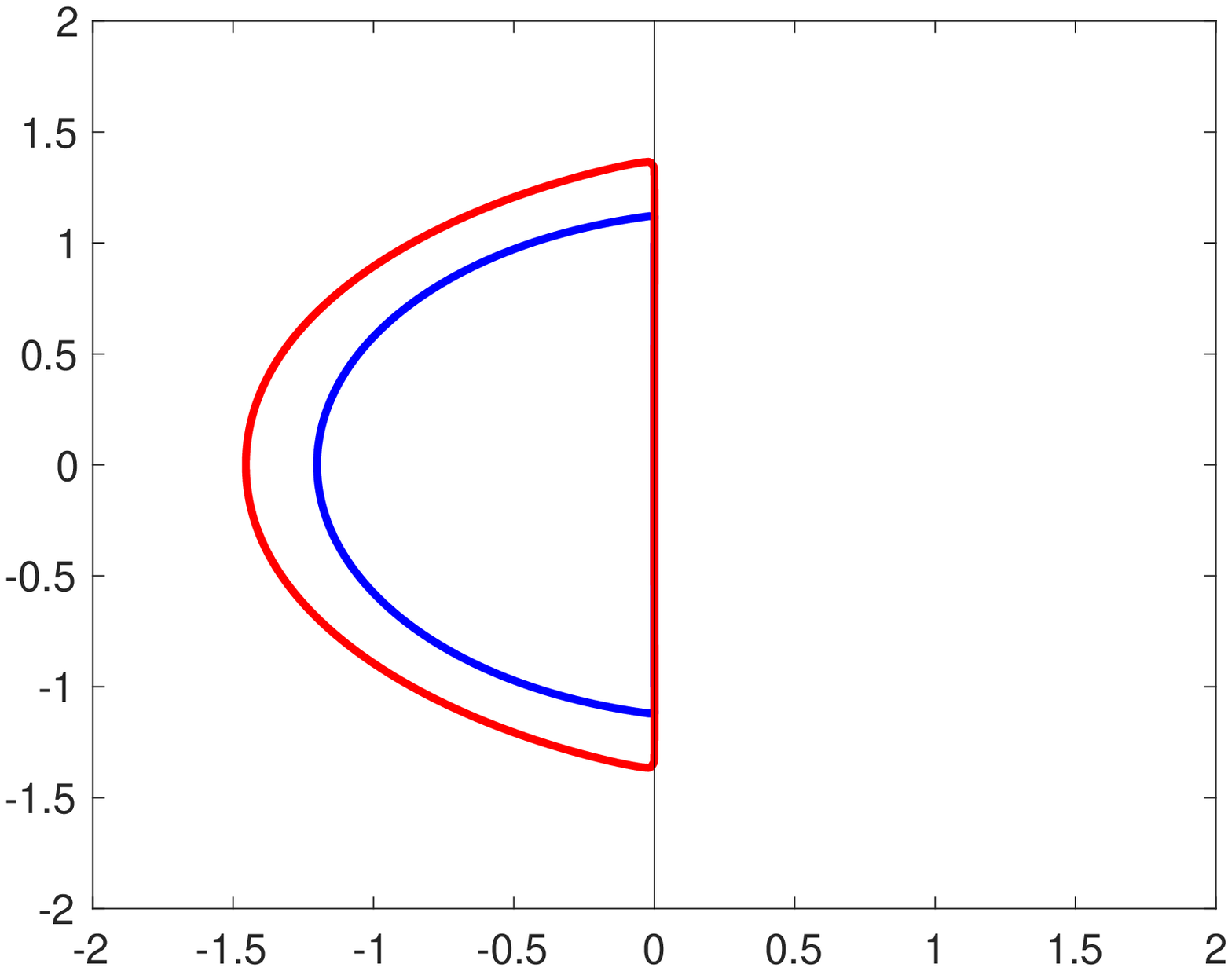} 
   \includegraphics[width=.485\textwidth]{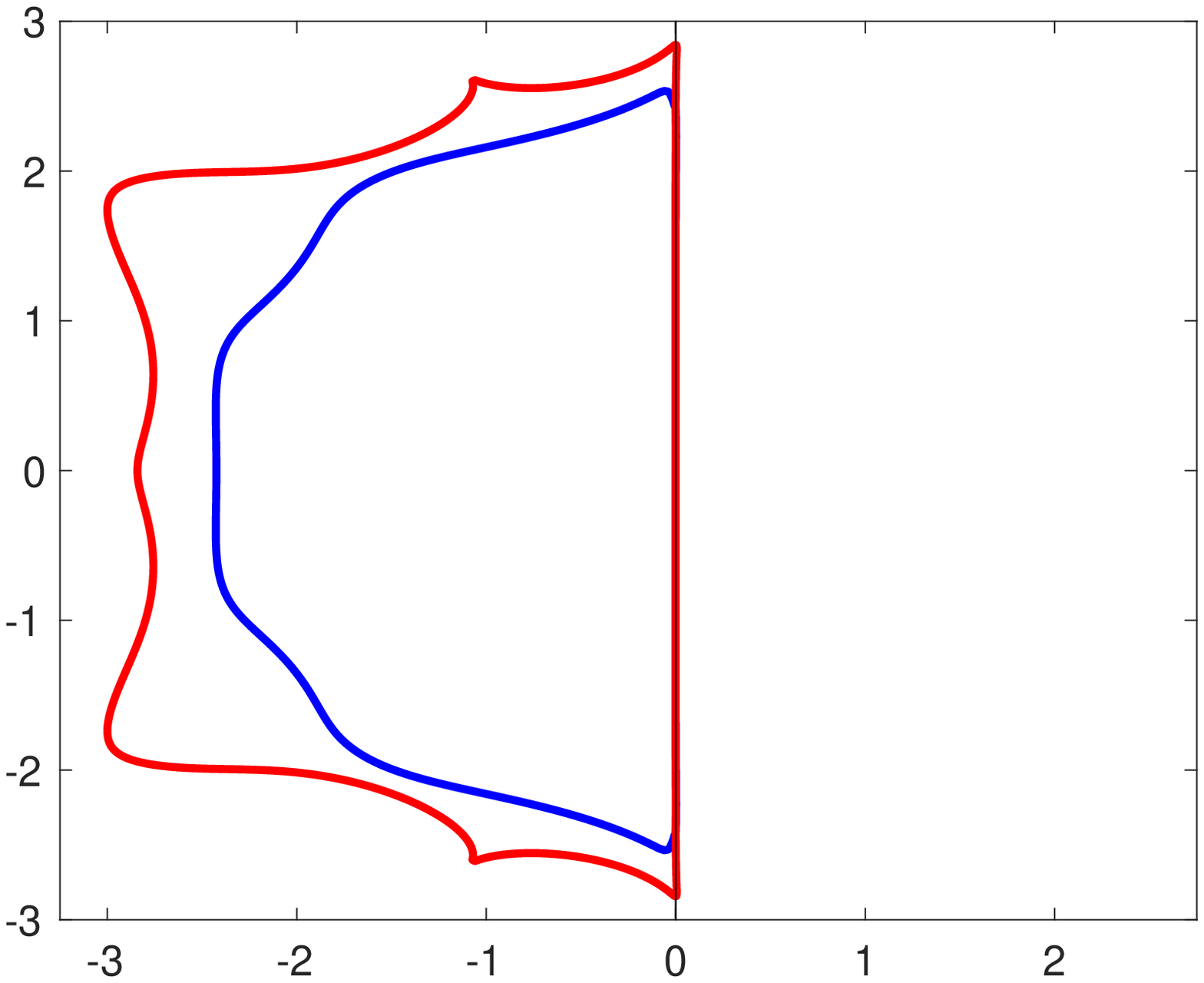} 
 \caption{\label{EISvsEIS+}  The linear stability region of the 
 EIS(s,P)$_2$ compared to the EIS+(s,P)$_2$ methods.
 On the left: the stability region of the eEIS(2,5)$_2$  in blue  and  eEIS+(2,5)$_2$   in red.
  On the right: the stability region of the
 and  eEIS(3,7)$_2$  in blue and  eEIS+(3,7)$_2$ in red.
 }
 \end{center}
\end{figure}

We now focus on the linear stability regions of some very high order EIS+(s,P)$_2$ methods.
These methods are further explored in Section \ref{sec:test} where we test these methods for convergence.
The coefficients of these methods are found in Appendix \ref{appendix:coef}.
Figure   \ref{linstabEIS+} shows the linear stability regions of the eEIS+(2,6)$_2$ (left),
 eEIS+(3,7)$_2$ (middle),  and  eEIS+(4,8)$_2$ (right) methods. 
 The axis limits are in all cases $(-s,s)$, to account for the increased cost 
of function evaluations that is proportional to the number of stages. 
The high order methods  eEIS+(3,7)$_2$,  and  eEIS+(4,8)$_2$ clearly have 
very favorable regions of linear stability.

\begin{figure}[h!]
\begin{center}
 \includegraphics[width=.32\textwidth]{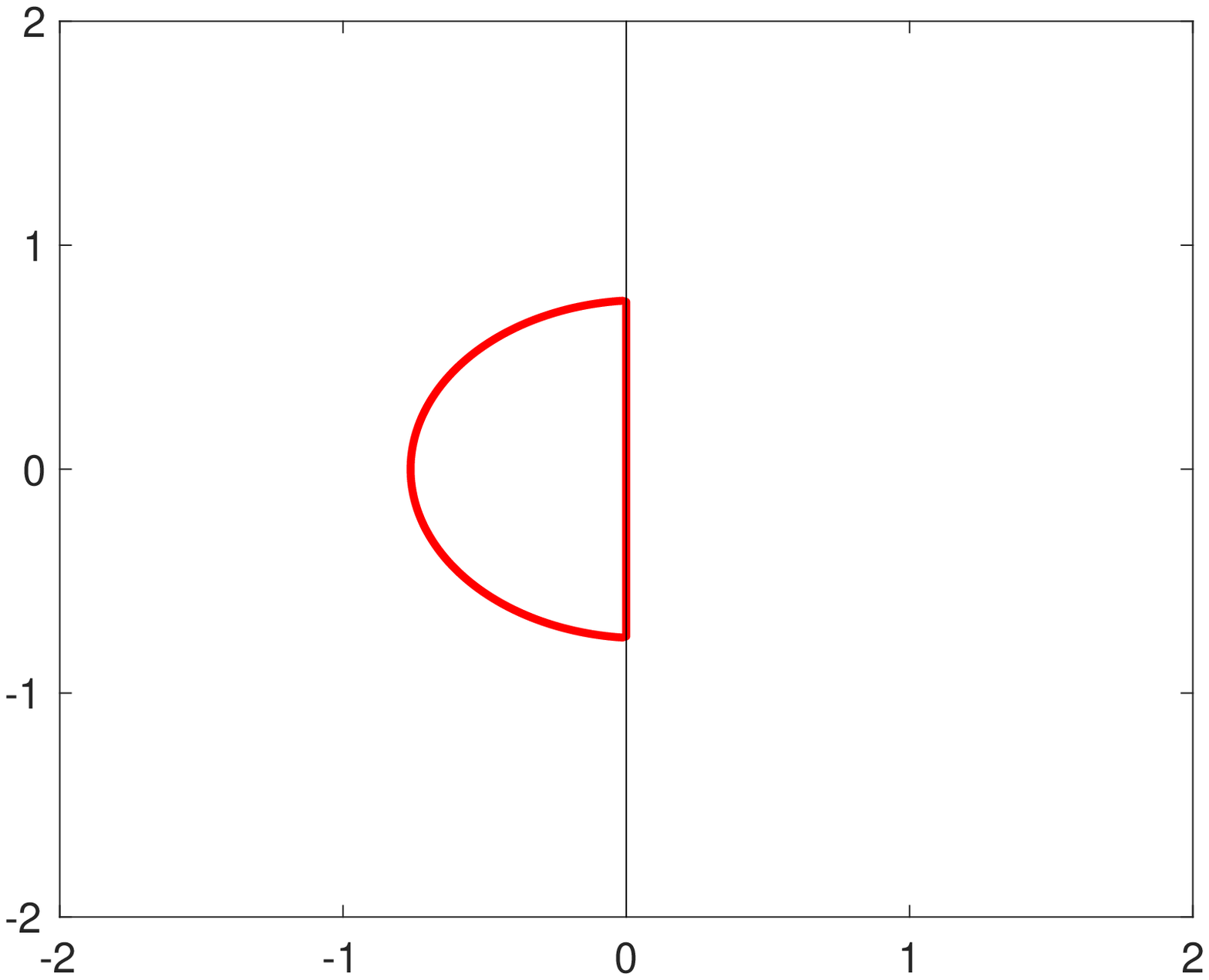} 
   \includegraphics[width=.32\textwidth]{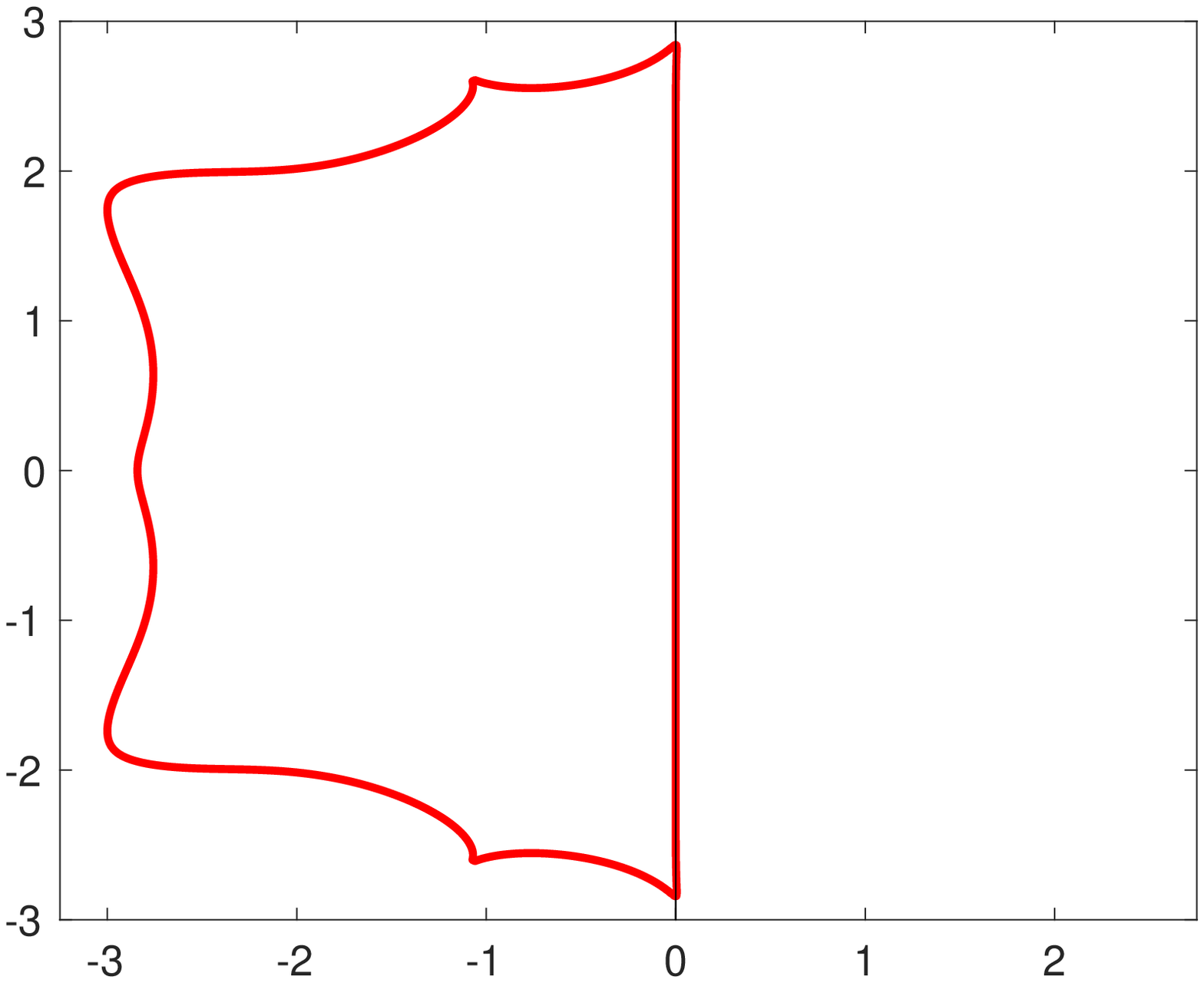} 
      \includegraphics[width=.32\textwidth]{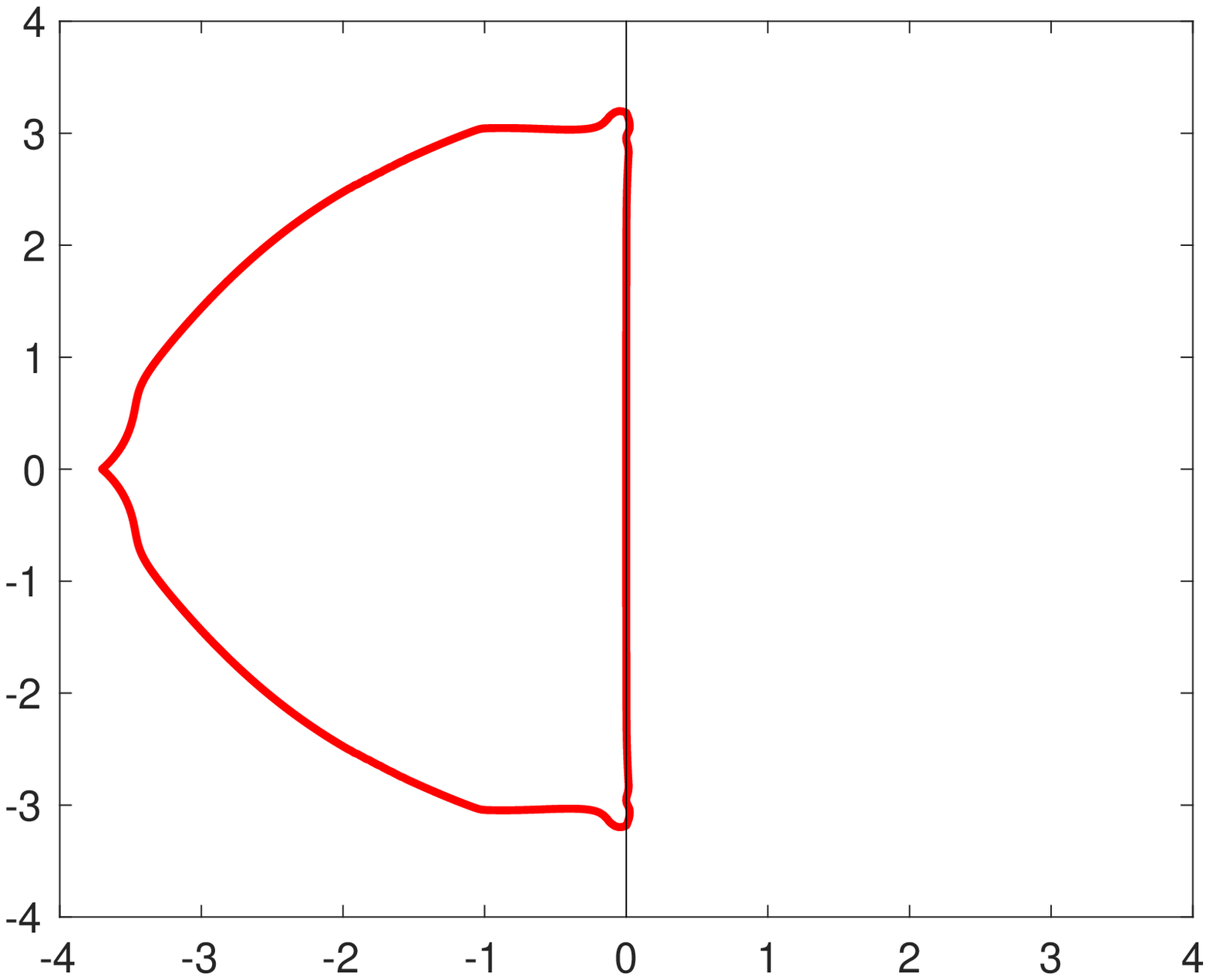} 
 \caption{\label{linstabEIS+}  
 The linear stability regions of the eEIS+(2,6)$_2$ (left),
 eEIS+(3,7)$_2$ (middle),  and  eEIS+(4,8)$_2$ (right) methods. 
 Note that both the x-axis and y-axis limits are proportional to the number of stages,
 so that the graph is on $(-s,s) \times(-s,s) $. This accounts  for the increased cost 
of function evaluations that is proportional to the number of stages. 
 }
 \end{center}
\end{figure}

\subsection{Explicit strong stability preserving EIS methods}  \label{SSP_methods} 
The solution to a hyperbolic conservation law  may develop sharp gradients or discontinuities, 
which requires careful handling to ensure that the numerical solution is stable.
For this purpose, a spatial discretization is carefully designed to satisfy some nonlinear stability 
properties (such as  total variation diminishing, maximum norm preserving,  or positivity preserving properties)
to ensure that it can handle the presence of a discontinuity. However, this spatial discretization method
must be paired with an appropriate time-stepping method to ensure preservation of these
nonlinear non-inner-product stability  properties \cite{SSPbook2011}. 
Such time-stepping methods are known as strong stability preserving  (SSP) time discretizations \cite{SSPbook2011}.

SSP time-stepping methods preserve the nonlinear  non-inner-product stability  properties
of the spatial discretization when coupled with forward Euler, under some time-step restriction
\begin{eqnarray} \label{FE}
 \|u^n +\Delta t F(u^n) \| \leq \| u^n  \| 
	 \;\;   \mbox{for}   \;\;  \Delta t \leq \Delta t_{FE}.
\end{eqnarray}
For two derivative methods, preserving the forward Euler condition is not enough: we need to augment this with a
condition that includes the second derivative. In this work we consider, as we did in \cite{MSMD-TS} the Taylor series condition
\begin{eqnarray}  \label{TS}
 \|u^n +  \Delta t F(u^n) +  \frac{1}{2} \Delta t^2 \dF(u^n) \| \leq \| u^n  \|  \;\;   \mbox{for}  \;\;    \Delta t \leq K \Delta t_{FE},
\end{eqnarray}
where $K$ is some constant.
(This condition is more general than the one we considered in  \cite{MSMD}) 
The explicit SSP two-derivative GLM  methods that we develop in this work can 
be decomposed into a convex combination of \eqref{FE} and \eqref{TS} with $K=1$, and thus
preserve their strong stability properties under the time-step restriction
\[ \dt \leq \sspcoef \Delta t_{FE}.\]

Using the optimization code we found a number of two-derivative error inhibiting 
strong stability preserving (SSP) methods with optimized
SSP coefficient $\sspcoef$. We denote by eSSP-EIS(s,P)$_2$ 
the methods that satisfy the order conditions to order $p=P-1$ and the  error inhibiting
condition \eqref{con1}. The methods  that satisfy the order conditions to order $p=P-2$
and satisfy the error inhibiting conditions \eqref{con1} -- \eqref{con3}  are denoted 
eSSP-EIS(s,P)+$_2$. We present three SSP error inhibiting methods in this section.

\noindent{\bf Explicit strong stability preserving eSSP-EIS(2,3)$_2$ with $\sspcoef =1.5$.} 
This method satisfies the truncation error conditions up to order $p=2$ and the EIS condition \eqref{con1} 
so that we obtain third order convergence ($P=3$). 
The coefficients of this method are rational and given by:
\[ \mD =  \frac{1}{16} \left( \begin{array}{ll}
7 & 9 \\ 
7 & 9 \\ 
\end{array}\right), \; \; 
\mA =  \frac{1}{8} \left( \begin{array}{ll}
2 & 3 \\ 
2 & 3 \\ 
\end{array}\right), \; \; 
\mAh =  \frac{1}{8} \left( \begin{array}{ll}
0 & 1 \\ 
0 & 1 \\ 
\end{array}\right), \; \;  \mR_{2,1}  = \frac{2}{3}, \; \;  \mRh_{2,1} = \frac{2}{9}.  \]
The remaining coefficients of $\mR$ and $\mRh$ are zero,  as this method is explicit.
It is interesting to note that the optimization procedure converged to this EIS method even when we requested
a method that satisfies the truncation error condition only up to order $p=1$ and  all 
the EIS+ conditions \eqref{con1} --  \eqref{con3}.

The SSP coefficient of this method is $\sspcoef=1.5$. We compare this to the SSP coefficient of the  standard 
third order GLM method with two steps and two stages given in \cite{msrk}, 
we denote this one-derivative method by eSSP-MSRK(2,2,3), indicating that it has two steps, two stages, and is order 3.
The SSP coefficient of the eSSP-MSRK(2,2,3) method is $\sspcoef \approx 0.73$.
Our two derivative error inhibiting method eSSP-EIS(2,3)$_2$ (with $\sspcoef=1.5$) 
has 4 function evaluations and the eSSP-MSRK(2,2,3) method 
has two function evaluations, so the  two method would seem to be equally efficient; 
however, if we consider applications in which the second derivative is available with no
additional cost, then the two-step method eSSP-EIS(2,3)$_2$ 
has an {\em effective SSP timestep} that is   twice as large.
 
\begin{table} 
\caption{A table of SSP coefficients of the error inhibiting methods found and the comparable methods in \cite{msrk}. 
 A dash denotes that no such method was obtained.
The asterisk indicates that the code searching for a method of the type eSSP-EIS+(s,P)$_2$
actually converged to a method of the type eSSP-EIS(s,P)$_2$.
\label{tab:ssp}}
\begin{center}
\begin{tabular}{|llccc|} \hline
s & P & eSSP-EIS(s,P)$_2$ & eSSP-EIS+(s,P)$_2$& eSSP-MSRK(s,s,P) \\
2 & 3 & 1.500 & * & 0.732 \\
2 & 4 & 0.990 & 1.000 &  --  \\
2 & 5 & -- & 0.555 & -- \\ \hline
3 & 4& 1.811 & 1.896  & 1.163\\
3& 5 & 1.369 &  1.548 & 0.638 \\
3 & 6 & 0.546  & 1.078 & 0.029 \\
3& 7 & -- & 0.129 & -- \\ \hline
\end{tabular}
\end{center}
\end{table}

We found a number of other SSP methods, and present their SSP coefficients in Table \ref{tab:ssp}. 
The first column gives the number of stages $s$, the second column gives the overall order $P$ that 
the solutions can be expected to be: for EIS methods this is $P=p+1$ and for EIS+ methods
this is $P=p+2$, which is the order after post-processing. 
We note that we did not find an  eSSP-EIS+(2,3)$_2$ method because the optimization scheme converged 
to the  eSSP-EIS(2,3)$_2$ method instead, and that we could not obtain an  eSSP-EIS(2,5)$_2$ 
or an eSSP-EIS(3,7)$_2$ method and believe they do not exist.

In Table  \ref{tab:ssp} we compare the SSP coefficients of the eSSP-EIS(s,P)$_2$ and eSSP-EIS+(s,P)$_2$
methods  to that of the comparable SSP GLM methods of \cite{msrk}. These methods are 
denoted eSSP-MSRK(s,s,P)  to indicate that they have $s$-steps and $s$-stages and are of order $P$.
This is a fair comparison in such cases where the derivative evaluation does not incur any additional cost.
This table shows that for a given $s$ and $P$, the  EIS+ methods have larger allowable SSP coefficients
and so are more efficient, assuming that the second derivative is available for no additional cost and  
that the cost of post-processing is negligible.   The coefficients of all the EIS and EIS+ methods in 
Table \ref{tab:ssp} can be downloaded from our   GitHub directory \cite{EISpp_MD_github}. 
 Below we list the coefficients of two featured methods.
 
 \smallskip

\noindent{\bf Explicit strong stability preserving eSSP-EIS+(2,4) with $\sspcoef =1.0$:} This method has coefficients
\[ \mD =   \left( \begin{array}{ll}
0.435605756635718  & 0.564394243364282 \\
   0.435605756635718 &  0.564394243364282 \\
\end{array}\right), \]
\[ \mA =   \left( \begin{array}{ll}
0.232303428413552  &  0.564394243364282 \\
0.216263460427852  &  0.564394243364282 \\
\end{array}\right), \]
\[ \mAh =  \left( \begin{array}{ll}
0.000000005124887  &  0.260081562620613 \\
0.000000001928255  &  0.146835746492061 \\
\end{array}\right), \]
where the non-zero elements of $\mR$ and $\mRh$ are 
\[ \mR_{2,1}  =  0.376253295127924, \; \; \mbox{and} \; \; \; 
 \mRh_{2,1} =  0.162082671864920  . \]
 The truncation error vector (needed for postprocessing) is 
 \[ \tau_3=  \left( \begin{array}{r}
 -0.063938362828511 \\
 0.049348339827035 \\
  \end{array}\right).
  \]

\smallskip

\noindent{\bf Explicit strong stability preserving eSSP-EIS+(3,6) with $\sspcoef =1.0782 $:} 
This method has coefficients
\[ \mD_{1,i}  =0.235787420033905, \; \;  \mD_{2,i}  = 0.332249926343388,\; \;  \mD_{3,i}  =0.431962653622707,\]

 \[\mA = \left( \begin{array}{ccc}
0.179040619183497  & 0  &  0.400647796399945 \\ 
   0.147616987633695 &   0.118289307755180   & 0.400647796399945 \\
   0.194101834261448  &  0.212027154638658   & 0.400647796399945 \\
\end{array}\right), \] 
\[  \mAh = \left( \begin{array}{ccc}
  0.032860477842919   &  0  &  0.068024553668439 \\
   0.024965463148830  &  0.034155124171981  &  0.021087452933654 \\
   0.011487692416560   & 0.092903917927740  &  0.124915188800131 \\
   \end{array}\right),
  \]
   \[\mR = \left( \begin{array}{ccc}
                 0           &        0       &            0 \\
   0.287524583705647           &        0         &          0 \\
   0.214948333287866  &  0.243023557774243         &           0 \\
\end{array}\right), \]
\[  \mRh = \left( \begin{array}{ccc}
                   0          &         0      &              0 \\
   0.133340336145235         &           0          &          0 \\
   0.050250968106130   & 0.112702859933545         &           0 .\\
   \end{array}\right),
  \]
   The truncation error vector (needed for postprocessing) is 
 \[ \tau_5= \left( \begin{array}{r}
 -0.010752778908703 \\
 -0.021534888908005 \\
 0.022433270953649 \\
 \end{array}\right).
 \]

We note that only the method itself is guaranteed to be SSP, we do not expect the  post-processed solution  to be SSP as well. 
The  post-processor is only designed to extract a higher order solution but not to preserve the strong stability properties. 
This does not typically pose a problem concern because preserving the nonlinear stability properties is generally only
important for the stability of the time evolution: once we reach the final time solution these properties
are no longer needed. If we still desire some properties that are destroyed by the post-processor we can opt to
use the non-post-processed solution instead.

\subsection{Implicit  A-stable  EIS methods} \label{implicit_methods} 
The implicit A-stable methods that we found treat both $F$ and $\dF$ implicitly. 
We tried to find methods that treat $F$ implicitly and $\dF$ explicitly, 
but were not able to find A-stable methods of this type. In this section we present two A-stable implicit methods 
iEIS+(s,P)$_2$ methods. Both $\mR$ and $\mRh$ are  diagonal matrices and so the method
can be implemented efficiently in parallel,  as each stage is computed independently. 
The coefficients of the methods as well as the $\ste{P-1}$  vector needed for  post-processing
are given below.

\smallskip

\noindent{\bf A-stable parallel-efficient iEIS+(2,4)$_2$:}

\[ \mD_{1,i} =0.594710614896760 , \; \; \;   \mD_{2,i} = 0.405289385103240,\]
\[ \mA = \left( \begin{array}{cc}
  -2.187376304427630  & -0.964459220078949 \\
  -1.117865907067007   & 2.067845436796621 \\
 \end{array}\right), \]
\[\mAh = \left( \begin{array}{cc}
   0.778080609332642  & -1.088765766927099 \\
  -2.898999040140121  &  1.440243113199464 \\
 \end{array}\right), \]
\[ \mR =\left( \begin{array}{cc}
   3.949190831954959     &              0 \\
                   0   & 0.347375777718766 \\
 \end{array}\right), \]
\[\mRh = \left( \begin{array}{cc}
 -2.706937237458932          &         0 \\
                   0  &  0.978108368826293 \\
 \end{array}\right). \]
    The truncation error vector (needed for postprocessing) is 
\[\tau_3 =  \left( \begin{array}{r}
-3.111010490530440  \\
4.565012136457357  \\
 \end{array}\right).\]

\noindent{\bf A-stable parallel-efficient iEIS+(3,5)$_2$:}

\[ \mD =\left( \begin{array}{ccc}
   0.439087264857344  & 0.700945256500558  & -0.140032521357901 \\
   0.439087264857344  & 0.700945256500558  & -0.140032521357901 \\
   0.439087264857344  & 0.700945256500558  & -0.140032521357901 \\
 \end{array}\right), \]

\[ \mA = \left( \begin{array}{ccc}
   2.507826539020301 &   3.279683213077780  & -1.170881137598611 \\
  -0.334032190141782  & -4.031402321497854  & 0.685583668720811 \\
  -1.750770284075905  & -4.999999998880823  & 3.295317723260540 \\
 \end{array}\right), \]
\[ \mAh = \left( \begin{array}{ccc}
   2.333968082671988   & 0.419378200972933  & -2.408406401605122 \\ 
  -2.145600247202041   & 0.897829295036851  & -0.721006948644857 \\
  -4.988816152192916   & 3.020756581381562  & -1.533772624102988 \\
 \end{array}\right), \]
\[ \mR= \left( \begin{array}{ccc}
 -3.756922019094389        &           0        &           0 \\
                   0   & 4.872890771657239         &          0 \\
                   0           &        0   &  4.981825821767937 \\
 \end{array}\right), \]
\[ \mRh =\left( \begin{array}{ccc}
   3.591518759368352           &        0         &          0 \\
                   0  & -2.760598976218027      &             0 \\
                   0         &          0    & -3.950356833416136 \\
 \end{array}\right). \]
   The truncation error vector (needed for postprocessing) is 
\[ \tau_4 =  \left( \begin{array}{r}
3.466008686399261 \\
-4.575755330149971 \\
-12.036302018622621 \\
 \end{array}\right). \]

\section{Numerical Results\label{sec:test}}
In this section we focus on testing the numerical methods found and presented in Section \ref{newEISMDmethods}.
We first verify the convergence properties of the methods presented in Sections \ref{explicit_methods} and \ref{implicit_methods},
on a system on nonlinear ODEs. We find that these methods achieve the desired convergence rates.
Next, we explore the behavior of the SSP methods presented in Section \ref{SSP_methods}
in terms of preserving the total variation diminishing  properties of spatial discretizations, 
and show that these methods behave as expected on a standard test-case (see \cite{MSMD-TS}).

\subsection{Convergence of high order explicit and implicit schemes}
We focus on verifying the convergence of the explicit high order methods eEIS+(2,6)$_2$,
eEIS+(3,7)$_2$),  and eEIS+(4,8)$_2$ presented in Section \ref{explicit_methods} and the implicit high order methods
iEIS+(2,4)$_2$ and iEIS+(3,5)$_2$ presented in Section \ref{implicit_methods} .
We show that these methods show the predicted order before and after post-processing on
a nonlinear system of ordinary differential equations.

\noindent{\bf Non-stiff Van der Pol oscillator:}
The nonlinear system of ODEs is given by
\[ \left( \begin{array}{l} y_1 \\ y_2 \end{array} \right)'
= \left( \begin{array}{c} y_2 \\ a(1-y_1^2) y_2 - y_1 \end{array} \right)
\]
with $a=2$ and  initial condition $\vy(0) = (2,0)^T$. We use the explicit methods eEIS+(2,6)$_2$,
eEIS+(3,7)$_2$, eEIS+(4,8)$_2$ to evolve this problem to the final time $T_f=3.0$ and postprocess the
solution at the final time as described in Section \ref{sec:postproc}.
We use three repeats of the truncation error 
vector ($m=3$) for  all methods except  the  eEIS$_2$+(2,6) method for which we used  $m=4$. (We note that
the  iEIS+(3,5)$_2$ method only requires $m=2$ but $m=3$ works well so we use it for consistency with the other
methods).

 In Figure \ref{VDPproblem} we show the square-root of the sum of squares of the errors  in  each component
  for different values of $\dt$. The slopes of these lines are computed using MATLAB's {\tt polyfit} function 
and are as follows:\\
\begin{center}
\begin{tabular}{|l|ccc|cc|}\hline
  & \multicolumn{3}{c|}{eEIS$_2$+(s,P) =}   &  \multicolumn{2}{c|}{iEIS$_2$+(s,P) =} \\
                           & (2,6)  & (3,7) & (4,8)  &  (2,4)  &  (3,5)  \\ \hline
without   post processing &4.7  & 5.8 & 7.0  &   3.0 &   3.9 \\
after  post processing  & 5.8 &   6.6 &7.7 & 4.0 & 5.0 \\ \hline
\end{tabular}
\end{center} 
\smallskip
\noindent Clearly, the orders without postprocessing are of order $P-1$ and after post-processing are of order $P$
for the eEIS+(s,P)$_2$  and iEIS+(s,P)$_2$ methods found in Section \ref{explicit_methods} and \ref{implicit_methods}, respectively .
This example verifies that  numerical solutions from the explicit and implicit error inhibiting  methods 
attain the expected orders of convergence with and without post-processing.
These examples confirm that the mathematical conditions and the methods we found work as
 expected in practice. 
 
\begin{figure}
    \centering
          \includegraphics[width=.475\textwidth]{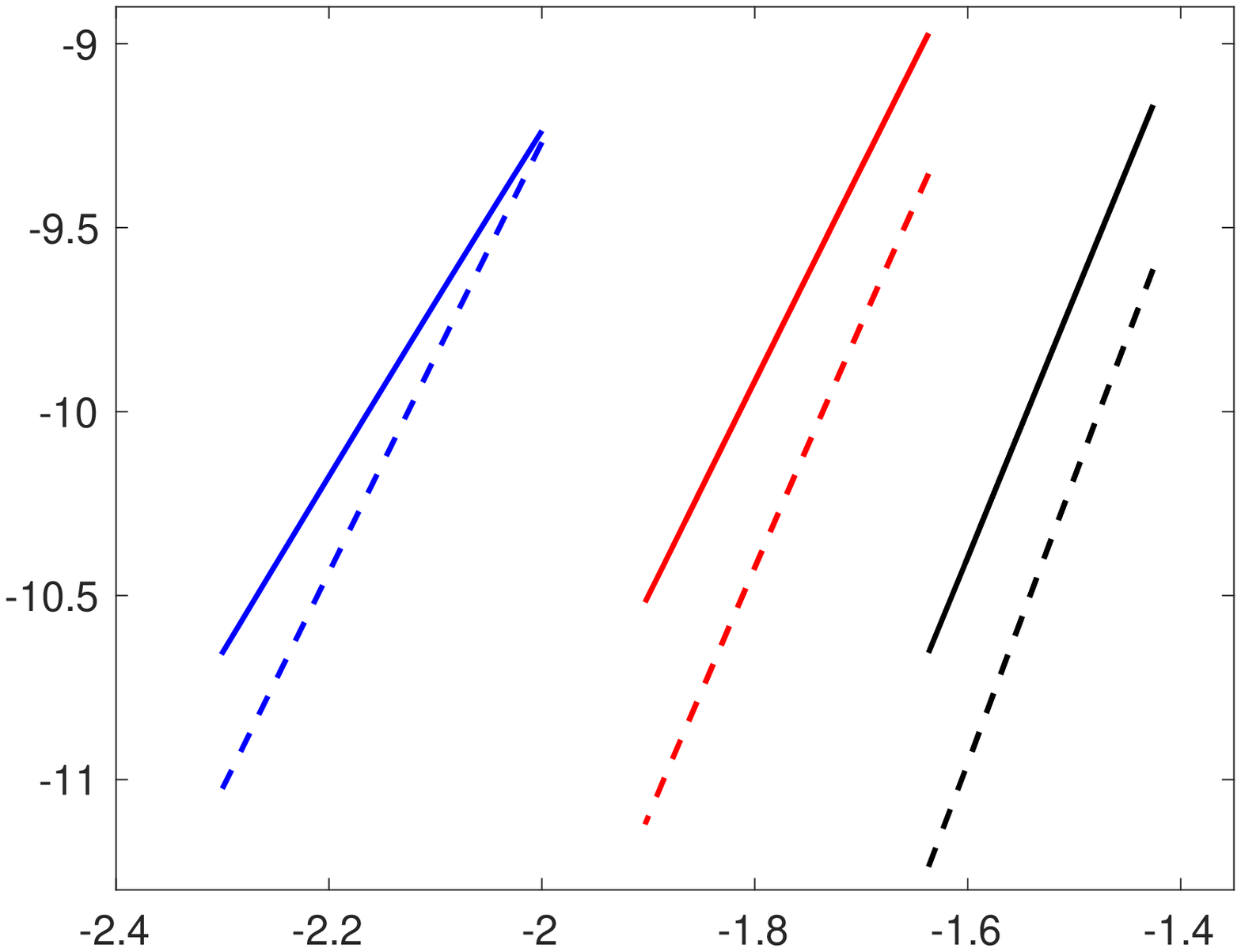} 
            \includegraphics[width=.475\textwidth]{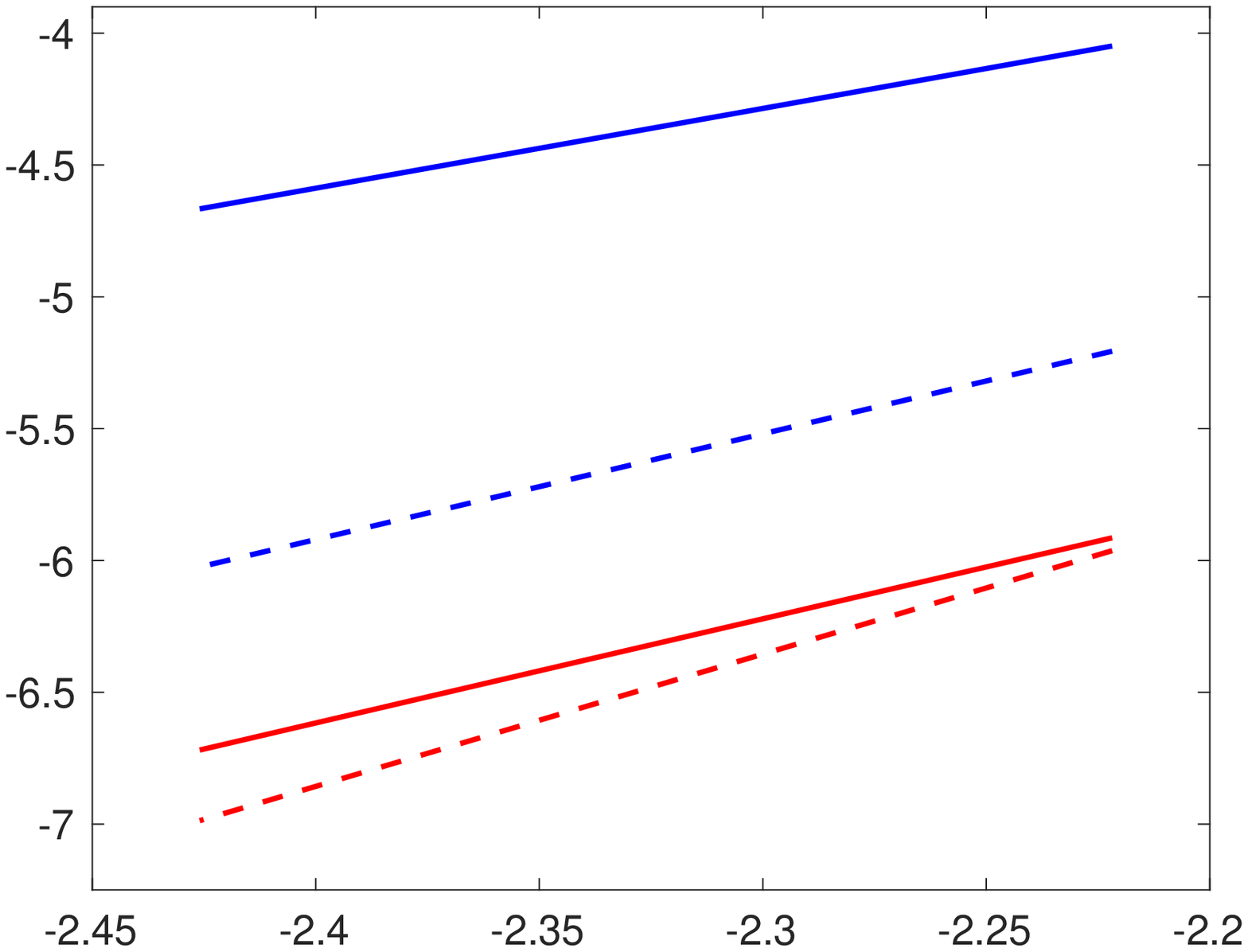} 
                     \caption{Convergence of the solution (in solid line)  and post-processed solution (dashed line)
                      from evolving the Van der Pol     system to time $T_f=3.0$  for the non-stiff Van der Pol system. 
                      On the y-axis is  the $log_{10}$ of the square-root of the  sum of squares  of the errors; 
                      on the x-axis is $log_{10}(\dt)$.
           Left: The explicit methods  eEIS$_2$+(2,6) (in blue), eEIS$_2$+(3,7) (in red), and eEIS$_2$+(4,8) 
           (in black). Right: The implicit methods  eEIS$_2$+(2,4) (in blue), eEIS$_2$+(3,5) (in red).
\label{VDPproblem} }
\end{figure}

\subsection{Comparison of explicit strong stability preserving schemes}
SSP time-stepping methods are useful for evolving in time hyperbolic PDEs with discontinuities.
We focused on finding SSP methods of the form \eqref{2Dmethods} that 
preserve the nonlinear non-inner-product stability properties of the spatial discretization methods 
when coupled with forward Euler and the Taylor series methods, as in \cite{MSMD-TS}.
The eSSP-EIS(s,P)$_2$ and eSSP-EIS+(s,P)$_2$ methods we presented in Section \ref{SSP_methods}
are tested here on a problem where the spatial discretization is total variation diminishing
 when coupled with forward Euler time-stepping and when coupled with a Taylor series method.
 To verify that the methods satisfy the SSP property we  examine the maximal rise in total variation
 when this problem is evolved forward with the  eSSP-EIS+(s,p)$_2$ scheme. We further 
 investigate the effect of post-processing on the total variation of this problem.

We consider the linear advection equation with step function initial conditions:
  \begin{align}
u_t + u_x & = 0 \hspace{.75in}
    u(0,x)  =
\begin{cases}
0, & \text{if } 0 \leq x \leq 1/2 \\
1, & \text{if } x>1/2 \nonumber
\end{cases}
\end{align}
on the domain $[-1,1]$ with periodic boundary conditions.
We use the fact that $u_{tt}= -u_{xt} = -u_{tx} = u_xx$ to approximate $\dF$.
For the spatial discretization we use a first order forward difference for the first and second derivatives with $200$ spatial points:
\[F(u^n)_j := - \frac{u^n_{j}-u^n_{j-1}}{\Delta x} \approx - U_x( x_j ), \]
\[ \dF (u^n)_j :=  \frac{u^n_{j}- 2 u^n_{j-1} + u^n_{j-2}}{\Delta x^2} \approx   U_{xx}( x_j ).
\]
These spatial discretizations satisfy:
\begin{description}
\item{\bf Forward Euler condition} \\
\[ u^{n+1}_j = u^n_j + \frac{\Delta t}{\Delta x} \left( u^n_{j-1} - u^n_j \right)\]
which  is  TVD for $ \Delta t \leq  \Delta x $, \\
\item{\bf Taylor series  condition}\\
\[ u^{n+1}_j = u^n_j + \frac{\Delta t}{\Delta x} \left( u^n_{j-1} - u^n_j \right) 
+ \frac{1}{2} \left(  \frac{\dt}{\dx} \right)^2 \left( u^n_{j-2} - 2 u^n_{j-1} + u^n_{j} \right) \]
 which is TVD for $\Delta t \leq  \Delta x $.
\end{description} 
So that $\DtFE= \Delta x $ and in this case we have $K=1$ in \eqref{TS}.

We evolve the solution 10 time-steps using the SSP methods eSSP-EIS(2,3)$_2$,
eSSP-EIS+(2,4)$_2$  and eSSP-EIS$+(3,6)_2$ for different values of $\dt$. 
At each time-level $y^n$ we compute the total variation of the solution at time $u^n$,  
\[ \| u^n \|_{TV} = \sum_{j} \left| u^n_{j+1} - u^n_j \right| . \]
The maximal rise in total variation (solid line) 
is graphed against $\lambda = \frac{\dt}{\dx}$ in Figure \ref{SSPtest}.
As we can see, the value of $\lambda$ for which the total variation begins to rise
is always greater that the guaranteed SSP value. 

\begin{figure}[t]
    \centering
          \includegraphics[width=.51\textwidth]{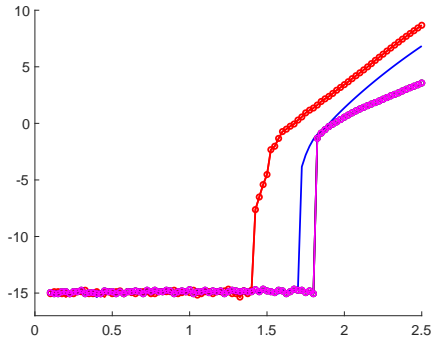} 
           \caption{SSP study: the maximal rise in total variation of the explicit SSP methods 
            SSP-EIS+(2,4)$_2$ (red line), eSSP-EIS(2,3)$_2$ (blue line), and
    eSSP-EIS+(3,6)$_2$ (magenta line) when used to evolve forward the solution in time
           of the linear advection equation with a step function initial conditions.
           On the y-axis is $log_{10}$ of the maximal rise in total variation, and 
           on the x-axis   the CFL number $\lambda$. 
           The circles mark the difference between the total variation of the un-processed solution 
           and that of the post-processed solution.             \label{SSPtest}  }
\end{figure}

We then post-process the solution from the 
eSSP-EIS+(2,4)$_2$  and eSSP-EIS+(3,6)$_2$ methods 
at the final time for all values of $\dt$,
and compute the difference between the total variation of the solution at the 
final time and the postprocessed solution
\[  \| u^n \|_{TV}  -  \| \widehat{u}^n \|_{TV}  \]
We see in  Figure \ref{SSPtest}  that the difference between the total variation of the solution at the  final time and the 
post-processed solution (circle markers) is minimal for  all the values of $\dt$ for which the maximal rise 
in total variation remains bounded.
In particular, for values of $\lambda$ for which the total variation has not started to rise, we
 observe that the maximal difference between the total variation of the
solution and the post-processed solution   remains very small $(\approx 10^{-15})$.
Although the post-processor is only designed to
extract a higher order solution but not to preserve the strong stability properties, it 
is reassuring that the total variation is not adversely impacted by the post-processor in this case.

\section{Conclusions}
  In this paper we extended the theory developed in \cite{EIS2pp} to the case of two-derivative GLMs.
  In \cite{EIS2pp} we showed, for a one-derivative GLM, the exact form of the error in the case that three error inhibiting conditions hold,
  and that this error is one order higher than expected by the truncation error analysis and the next error term is 
  multiplied by a truncation error vector with a known form. Furthermore, we designed a post-processor that
  can eliminate the next term in the error, thus extracting an addition order of accuracy.
  In this work we extend this theory to two-derivative GLMs of the form \eqref{2Dmethods} and prove that under the sufficient conditions
  \eqref{con1}--\eqref{con3} we have an error of the form  \eqref{Error_form}. This form is very exactly the same as the form in 
  \cite{EIS2pp}, except that the value of $\ste_{p+1}$ is different for the two-derivative GLMs than for the one-derivative GLMs.
  For this reason, the construction of a post-processor (in Section \ref{sec:postproc}) follows immediately from the 
  theory in  \cite{EIS2pp}. 
  
  Using this theory, we programmed an optimization code in {\sc Matlab} that uses the order conditions and error inhibiting conditions
  as constraints and aims to maximize some stability properties. We generated a number of methods that satisfy the order conditions
  and the error inhibiting condition   \eqref{con1} or all the error inhibiting conditions \eqref{con1}--\eqref{con3}. 
  The coefficients of all the methods can be downloaded on our GitHub repository \cite{EISpp_MD_github}.  The explicit methods found by the optimizer
  either have favorable regions of linear stability or are SSP in the sense of \cite{MSMD-TS}. The implicit methods
  are A-stable. We test a selection of these methods for convergence, and show that they 
  demonstrate the predicted order before and after post-processing on
a nonlinear system of ordinary differential equations.
  We test the SSP methods for the preservation of the  total variation diminishing properties on a PDE with discontinuities, and show
  that they produce TVD results beyond the guaranteed SSP time-step, and that furthermore post-processing does not adversely 
  affect the total variation. These results validate the methods we found, as well as the theory that we developed.

  \appendix
  
  \section{Coefficients of some methods}
\label{appendix:coef}
In this appendix we list the coefficients of some of the  methods mentioned in Section 4.1.
The coefficients for all the methods in this paper can be downloaded from our GitHub repository \cite{EISpp_MD_github}.

\smallskip

\noindent{ \bf Explicit two-stage third-order EIS method (eEIS(2,3)$_2$):} 
\[ \mD_{1,i} =  1.347635863512091, \; \; \; 
\mD_{2,i} = -0.347635863512091,\]
\[ \mA = \left( \begin{array}{cc}
   1.110588320380528   & 0.206278390370703 \\
   1.160801319467423   & 0.191968442856969 \\
 \end{array}\right), \]
\[\mAh = \left( \begin{array}{cc}
   0.376508598017949   & 0.079881117612918 \\
   0.424704932282709   & 0.083778591655645 \\
 \end{array}\right), \]
\[ \mR = \left( \begin{array}{cc}
                   0         &          0 \\
   0.875587228946215        &           0 \\
 \end{array}\right), \; \; \; 
 \mRh = \left( \begin{array}{cc}
                   0         &          0 \\
   0.412259887079832          &         0 \\
 \end{array}\right), \]

\noindent{ \bf Explicit two-stage fifth-order EIS+ method (eEIS+(2,5)$_2$):}
\[ \mD_{1,i} =  0.500023658051142 , \; \; \; 
\mD_{2,i} =    0.499976341948858 , \]
\[ \mA = \left( \begin{array}{cc}
   0.627069692131650 &  0.151022064558538 \\
   0.709712162750524  & 0.848963643214302 \\
\end{array}\right), \]
\[ \mAh = \left( \begin{array}{cc}
   0.058142153689242 &  0.325582994094698 \\
   0.108273930132603  & 0.477624731406111 \\
 \end{array}\right), \]
\[ \mR = \left( \begin{array}{cc}
             0        &           0 \\
  -0.336746561995068       &            0 \\
 \end{array}\right), \; \; \; 
\mRh = \left( \begin{array}{cc}
                   0          &         0 \\
   0.367133756538675         &          0 \\
 \end{array}\right).\]
    The truncation error vector (needed for postprocessing) is 
\[ \ste_4 = (-0.039533847641586  ,   0.039537588993770 )^T.\]

\noindent{ \bf Explicit two-stage sixth-order EIS+ method (eEIS+(2,6)$_2$):}
\[ \mD_{1,i} =  0.193021555206000, \; \; \; 
\mD_{2,i} =    0.806978444794000 , \]
\[ \mA = \left( \begin{array}{cc}
   1.089589263420254  & -0.469532861646008 \\
   1.011690204056872   & 1.112307786855907 \\
    \end{array}\right)  \]
\[\mAh = \left( \begin{array}{cc}
   0.196914195858807 &   0.434709438834146 \\
   0.130811273979010  & 0.871687677021200 \\
 \end{array}\right), \]
\[ \mR = \left( \begin{array}{cc}
                   0          &         0 \\
  -1.033119102271808        &           0 \\
\end{array}\right), \; \; \; 
\mRh = \left( \begin{array}{cc}
                   0        &           0 \\
   0.499137031946415            &       0 \\
 \end{array}\right).\]
    The truncation error vector (needed for postprocessing) is 
\[ \ste_5= ( -0.037857689452761 , 0.009055198613815)^T. \]

\noindent{ \bf Explicit three-stage seventh-order EIS+ method  (eEIS+(3,7)$_2$):}

\[ \mD_{1,i} = 1.581021525561460  , \; 
\mD_{2,i} = -0.598751979308602   , \; 
\mD_{3,i} =0.017730453747142,\]
\[ \mA = \left( \begin{array}{ccc}
   0.931591460185742   & 0.379244369981835  & -0.172141957956410 \\
   0.938547162180577   & 0.508131122095280  & -0.363857858559788 \\
   0.504648760586788   & 1.046850936001111  & -0.659275924405796 \\
 \end{array}\right), \]
\[ \mAh = \left( \begin{array}{ccc}
   0.057154143906362   & 0.302522642478094   & 0.175689200743141 \\
   0.045099335357263   & 0.359020777972142   & 0.164798140168151 \\
  -0.060217523878309   & 0.456569929293375  & -0.005615338892051 \\
 \end{array}\right), \]
\[ \mR = \left( \begin{array}{ccc}
               0        &           0             &      0 \\
   0.307438691150295         &          0          &          0 \\
   1.789973573982305  & -0.870575633439973           &         0 \\
 \end{array}\right), \]
\[ \mRh = \left( \begin{array}{ccc} 
            0       &            0        &           0 \\
   0.038804362951013         &          0           &        0 \\
   0.227157707727078   & 0.276283023303938      &             0 \\
       \end{array}\right). \]
          The truncation error vector (needed for postprocessing) is 
\[ \ste_6 =  ( -0.003599790543666,  -0.012406980352919,   -0.097987210664809)^T.\]

\smallskip
  
\noindent{ \bf Explicit four-stage eighth-order EIS+ method (eEIS+(4,8)$_2$):}
\[ \mD_{1,i} =     1.126765222628176, \; \; \;
\mD_{2,i} =   0.808129178515260, \]
\[  \mD_{3,i} =  -0.107647150078402, \; \; \;
\mD_{4,i} = -0.827247251065033 ,\]
{\footnotesize
\[ \mA = \left( \begin{array}{cccc}
  0.567574025309926   & 0.723999455772069   & 0.208196137734782   & 0.023532165559543 \\
   0.749691669482323   & 0.430151531239573  &  0.359568096205409  & -0.030974711893773 \\
   0.602555996794216   & 0.745759221902972  & 0.048559187429251  & -0.267889537378177 \\
   1.051588361923041  & -0.047355340428569  & 0.863960642835203   & 0.214102220881218 \\
 \end{array}\right), \]
 \[ \mAh = \left( \begin{array}{cccc}
   0.041975696597772  &  0.205746598967380   & 0.137652258393657   & 0.039122406247340 \\
   0.064927843091523  & 0.213465637934016    & 0.160720650985361   & -0.047428374982532 \\
   0.056975020786010  &  0.171669459177575   & 0.226994033551341   & -0.021617692260293 \\
   0.095018403341495  &  0.263066907087928   & 0.147903147440657   & -0.036525606967693 \\
 \end{array}\right), \]}
 \[ \mR = \left( \begin{array}{cccc}
                   0           &        0           &        0            &        0 \\
   0.296825313241825         &          0          &          0        &           0 \\
   0.379857836431130  &  0.610459020171445   &                0       &             0 \\
   0.079086170545983  &  0.114409044614819   & 0.077980998192235    &               0 \\
 \end{array}\right), \]
  \[ \mRh = \left( \begin{array}{cccc}
                   0        &           0       &             0        &            0 \\
   0.095598816350501          &         0           &        0           &        0 \\
  -0.143446089841412  &  0.076113483149991        &           0        &           0 \\ 
   0.309290513515929  &  0.063106409144583   & 0.076129207423402        &            0 \\
 \end{array}\right). \]
   The truncation error vector (needed for postprocessing) is 
  \[ \ste_7=  \left( \begin{array}{c}
   -0.000997109517747 \\
  -0.006485724807936 \\ 
  -0.023117224006582 \\
  -0.004685791946531 \\
 \end{array}\right) . \]

\vspace*{0.75in}

{\bf Acknowledgment.} 
This publication is based on work supported by  AFOSR grant FA9550-18-1-0383 and ONR-DURIP Grant N00014-18-1-2255. A part of this research is sponsored by the Office of Advanced Scientific Computing Research; US Department of Energy, and was performed at the Oak Ridge National Laboratory, which is managed by UT-Battelle, LLC under Contract no. De-AC05-00OR22725. This manuscript has been authored by UT-Battelle, LLC, under contract DE-AC05-00OR22725 with the US Department of Energy. The United States Government retains and the publisher, by accepting the article for publication, acknowledges that the United States Government retains a non-exclusive, paid-up, irrevocable, world-wide license to publish or reproduce the published form of this manuscript, or allow others to do so, for United States Government purposes.

\end{document}